\documentclass[12pt]{article}

\usepackage{overpic,graphicx}
\usepackage{amssymb,amsfonts,amsmath,amsthm}
\usepackage{animate}
\usepackage{color}
\usepackage[english]{babel}
\usepackage[latin1]{inputenc}
\usepackage{times}
\usepackage{geometry}
\usepackage[breaklinks,pdfstartview=FitH,
colorlinks,linktocpage,bookmarks=false]{hyperref}

\newtheorem{thm}{Theorem}[section]

\newtheorem{lem}[thm]{Lemma}
\newtheorem{rem}[thm]{Remark}
\newtheorem{exmp}[thm]{Example}
\newtheorem{problem}[thm]{Problem}

\newtheorem{cor}[thm]{Corollary}

\newtheorem{lemma}[thm]{Lemma}

\newtheorem{mainthm}[thm]{Main Theorem}
\newtheorem{splitting-lemma}[thm]{Splitting Lemma}
\def\re{\mathrm{Re}}
\def\im{\mathrm{Im}}

\DeclareMathOperator{\id}{Id}

\def\bar#1{\overline{#1}}

\def\det{\mbox{det}}

\def\deg{\mathop{\mathrm{deg}}\nolimits}
\def\const{\mathrm{const}}

\def\R{{\mathbb R}}

\def\H{{\mathbb H}}

\long\def\comment#1\endcomment{}

\begin{document}


\title{Surfaces containing two circles through each point}

\author{M. Skopenkov, R. Krasauskas}
\date{}

\maketitle


\begin{abstract}
We find all analytic surfaces in $3$-dimensional Euclidean space such that through each point of the surface one can draw two transversal circular arcs fully contained in the surface (and analytically depending on the point). The search for such surfaces traces back to the works of Darboux from XIXth century. We prove that such a surface is an image of a subset of one of the following sets under some composition of inversions:

- the set $\{\,p+q:p\in\alpha,q\in\beta\,\}$, where $\alpha,\beta$ are two circles in $\mathbb{R}^3$;

- the set $\{\,2\frac{p \times q^{\hspace{3pt}}}{|p+q|^2}:p\in\alpha,q\in\beta,p+q\ne 0\,\}$, where $\alpha,\beta$ are circles in the unit sphere~${S}^2$;

- the set $\{\,(x,y,z): Q(x,y,z,x^2+y^2+z^2)=0\,\}$,
where $Q\in\mathbb{R}[x,y,z,t]$ has degree $2$ or $1$.


The proof uses a new factorization technique for quaternionic polynomials.

\smallskip

\noindent{\bf Keywords}: Moebius geometry, circle, quaternion, Pythagorean n-tuple, polynomials factorization

\noindent{\bf 2010 MSC}: 51B10, 13F15, 16H05
\end{abstract}

\footnotetext[0]{
This is an updated version of an article published in Mathematische Annalen. The authenticated published version is available online at: http://dx.doi.org/10.1007/s00208-018-1739-z.
The article was prepared within the framework of the Academic Fund Program at the National Research University Higher School of Economics (HSE) in 2015-2016 (grant No 15-01-0092) and supported within the framework of a subsidy granted to the HSE by the Government of the Russian Federation for the implementation of the Global Competitiveness Program. The first author was also supported by the President of the Russian Federation grant MK-6137.2016.1, ``Dynasty'' foundation, and the Simons--IUM fellowship. The second author was partially supported by the Marie-Curie Initial Training Network SAGA, FP7-PEOPLE contract PITN-GA-214584.
}

\section{Introduction}\label{s:intro}

\begin{tabular}{p{0.23\textwidth}p{0.71\textwidth}}
& \textit{
Dedicated to the last real scientists, searching only for the truth, not career, not glory, not pushing forward their own field or students
}
\end{tabular}

\bigskip
For which surfaces in $3$-dimensional Euclidean space, through each point of the surface one can draw two transversal circular arcs fully contained in the surface?
This is a question which simply must be answered by mathematicians
because of a natural statement and obvious architectural motivation --- recall Shukhov's hyperboloid structures.
It 
remained open in spite of many partial advances 
starting from the works of Darboux from the XIX century.
The present paper gives a complete answer.

\begin{mainthm}\label{mainthm}
If through each point of an analytic surface in $\mathbb{R}^3$ one can draw two transversal circular arcs fully contained in the surface (and analytically depending on the point) then some composition of inversions takes
the surface to a subset of one of the following sets (see Video~\ref{movie}):
\begin{itemize}
\item[(E)]\label{ETS} the set $\{\,p+q:p\in\alpha,q\in\beta\,\}$,  where $\alpha,\beta$ are two circles in $\mathbb{R}^3$;
\item[(C)]\label{CTS} the set $\{\,2\frac{p \times q^{\hspace{3pt}}}{|p+q|^2}:p\in\alpha,q\in\beta,p+q\ne 0\,\}$, where $\alpha,\beta$ are two circles in the unit sphere ${S}^2$;
\item[(D)]\label{DC} the set $\{\,(x,y,z): Q(x,y,z,x^2+y^2+z^2)=0\,\}$,
where $Q\in\mathbb{R}[x,y,z,t]$ has degree $2$ or $1$.
\end{itemize}
\end{mainthm}


\begin{figure}[!t]
\begin{center}
\vspace{-0.3cm}
\begin{tabular}{c}
\begin{tabular}{c}
\animategraphics[autoplay,loop,height=2.5cm]{7}{euc_A}{0}{30}
\end{tabular}
\hspace{-0.5cm}
\begin{tabular}{c}
\animategraphics[autoplay,loop,height=2.5cm]{7}{euc_B}{0}{30}
\end{tabular}
\\[-0.8cm]
\tiny{(E)}
\end{tabular}
\begin{tabular}{c}
\begin{tabular}{c}
\animategraphics[autoplay,loop,height=2.0cm]{7}{ani-S2-}{0}{71}
\end{tabular}
\hspace{-0.5cm}
\begin{tabular}{c}
\animategraphics[autoplay,loop,height=2.0cm]{7}{ani-S1-}{0}{71}
\end{tabular}
\\[-0.3cm]
\tiny{(C)}
\end{tabular}
\begin{tabular}{c}
\begin{tabular}{c}
\includegraphics[width=0.15\textwidth]{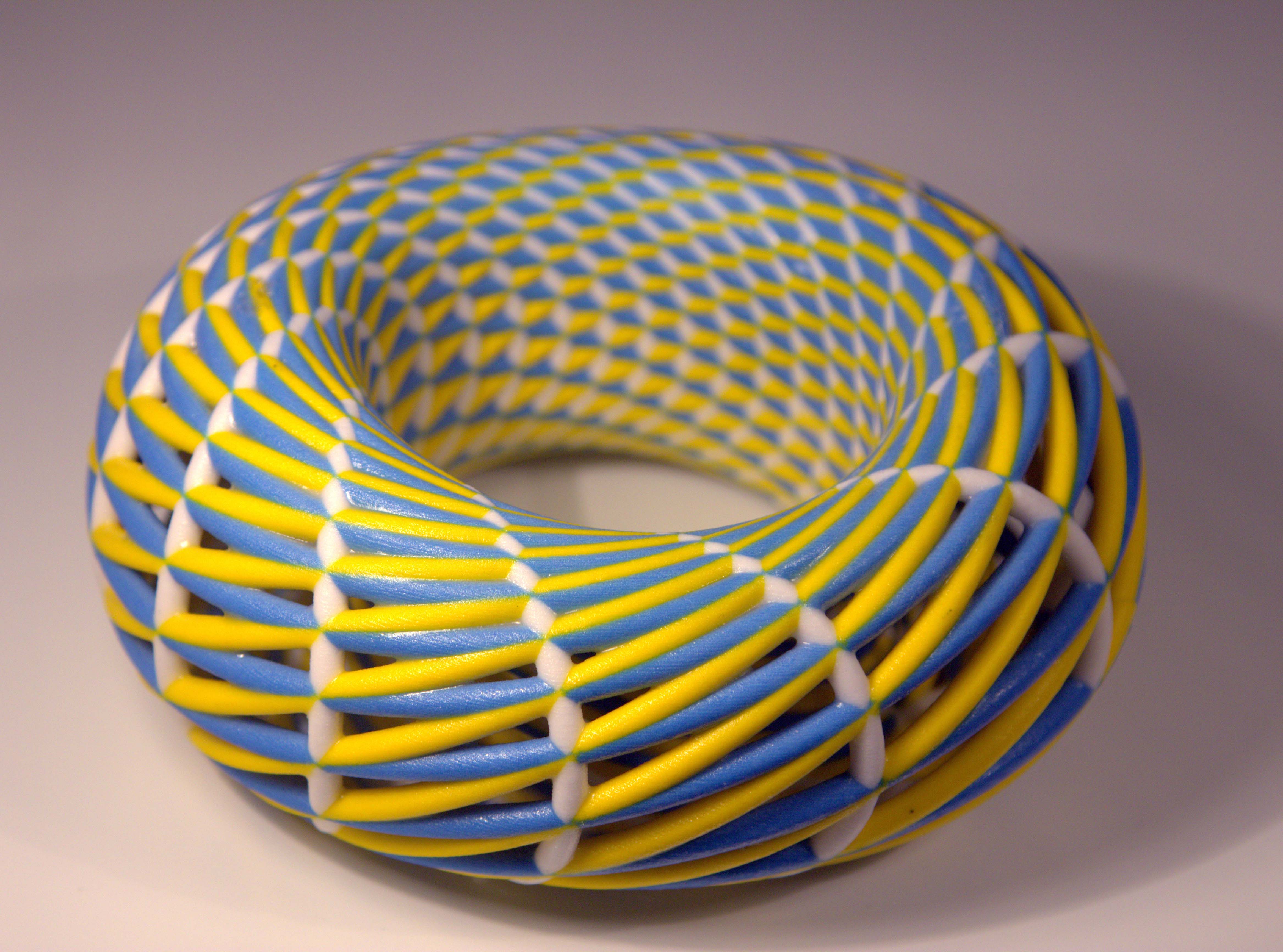}
\end{tabular}	
\\[-0.2cm]
\tiny{(D)}
\end{tabular}
\end{center}
\vspace{-0.6cm}
\renewcommand{\figurename}{Video}
\caption{
Euclidean (E) and 
Clifford  (C) translational surfaces, and a Darboux cyclide (D) \cite{Lubbes-14}.}
\label{movie}
\vspace{-0.4cm}
\end{figure}


Here $p,q$ denote both points in $\mathbb{R}^3$ and their radius-vectors, $p\times q$ denotes the cross product, and $S^2=\{\,p\in\mathbb{R}^3:|p|=1\,\}$. An \emph{analytic surface} in $\mathbb{R}^3$ is the image of an injective real analytic map of a planar domain into $\mathbb{R}^3$ with nondegenerate differential at each point. (Thus informally the theorem is ``local'' in a sense.) A circular arc \emph{analytically depending} on a point is a real analytic map of an analytic surface to the real analytic variety of all circular arcs in $\mathbb{R}^3$. (We have two such maps in the theorem.)  The analyticity is not really a restriction;  see the paragraph before Problem~\ref{pr-curved-chessboard} below.


\subsection*{Background}
The search for surfaces 
containing $2$ circles or lines through each point traces back to 
XIXth century. 
Basic examples --- a one-sheeted hyperboloid and
a nonrotational ellipsoid --- are discussed in Hilbert--Cohn-Vossen's ``Anschauliche Geometrie''. There (respectively, in \cite{NS11}) it is proved  that a surface containing $2$ lines (respectively, a line and a circle) through each point is a quadric or a plane. A torus contains $4$ circles through each point: a ``meridian'', a ``parallel'', and two Villarceau circles.



All these examples are particular cases of a \emph{Darboux cyclide}, which is by definition set~(D) in Theorem~\ref{mainthm}.
It is the stereographic projection of the intersection of $S^3$ with another $3$-dimensional quadric in $\mathbb{R}^4$ \cite[\S2.2]{PSS11}.
Almost every Darboux cyclide contains at least $2$ circles through almost every point \cite{coolidge:1916,PSS11,Takeuchi-00}; see Remark~\ref{rem-reciprocal}
for a list of exceptions.
Conversely, Darboux has shown that 
$10$ circles through each point guarantee that an analytic 
surface is a Darboux cyclide. 
This result has been improved over the years: in fact already $3$, or $2$ orthogonal, or $2$ cospheric circles are sufficient for the same conclusion
\cite[Theorem~3]{Niels-13},
\cite[Theorem~1]{ivey:1995},
\cite[Theorem~20 in p.~296]{coolidge:1916};
  cf. \cite[Theorems~3.4, 3.5]{
NS11}.
Hereafter $2$ circles are 
\emph{cospheric}, if they are contained in one 
sphere or plane.

Pottmann noticed that 
a \emph{Euclidean translational surface}~(E) contains $2$ circles through each point but is not a Darboux cyclide for generic $\alpha,\beta$ \cite[Example~3.9]{NS11}. The same is true for \emph{Clifford translational surface} found by Zub\.e (private communication), which is by definition set~(C) in Theorem~\ref{mainthm}.
It is the stereographic projection of the surface $\{\,p\cdot q:p\in\alpha,q\in\beta\,\}$ in $S^3$, where the unit sphere $S^2$ containing $\alpha$ and $\beta$ is identified with the set of pure imaginary unit quaternions; see~Lemma~\ref{l-clifford}.
The latter surface for arbitrary circles $\alpha,\beta$ in $S^3$ (rather than in $S^2$) is called a \emph{Clifford translational surface in $S^3$}.
Topologically it can be, e.g., an immersed torus not regularly homotopic to a torus of revolution.
It can have degree up to~$8$. Each degree $8$ surface in $S^3$ containing a \emph{great} circle and another circle through each point is Clifford translational
\cite[Corollary~2a]{Lubbes-14}.

More examples can be obtained from these translational surfaces by \emph{M\"obius transformations}, i.e., compositions of inversions. 
Euclidean and Clifford translational surfaces are not M\"obius transformations of each other \cite[Corollary~3]{Lubbes-14}. For related transformations taking lines to circles see \cite{Timorin-07}.




Surfaces containing $2$ circles through each point are particular cases of surfaces containing $2$ conic sections through each point.
The latter have been classified by Brauner, Schicho, Lubbes \cite{Brauner, schicho:2001, Lubbes-14}. In the particular case of so-called supercyclides 
a classification was given by Degen \cite{Degen-86}. The Schicho classification is up to a weak equivalence relation (namely, up to linear normalization) and does not allow to extract the surfaces containing $2$ circles through each point.
(We do not define the linear normalization because it is not used in the present paper.)
But his results reduce finding such surfaces in $\mathbb{R}^n$ to solving the equation
\begin{equation}\label{eq-pythagorean-n}
X_1^2+\dots+X_{n+1}^2=X_{n+2}^2
\end{equation}
in polynomials $X_{1},\dots,X_{n+2}\in\mathbb{R}[u,v]$ of degree at most $2$ in each of the variables $u$ and $v$; see Theorem~\ref{haupt}. Almost every such ``Pythagorean $(n+2)$-tuple'' of polynomials defines a 
surface $X_1(u,v):\dots:X_{n+2}(u,v)$ in $S^n$ (in homogeneous coordinates) containing $2$ 
circles $u=\mathrm{const}$ and $v=\mathrm{const}$ through almost every point. 
Eq.~\eqref{eq-pythagorean-n} gives a system of $25$ quadratic equations on $9(n+2)$ coefficients of the polynomials, hence it is not directly accessible for a computer analysis.

Usually equations like~\eqref{eq-pythagorean-n} are solved using \emph{factorization} --- recall parametrization of Pythagorean triples and Euler's approach to the Fermat Last Theorem \cite{Postnikov}. Pythagorean $3$- and $4$-tuples of real polynomials were described in \cite[Theorem~2.2]{Dietz-etal} using that $\mathbb{C}[u,v]$ is a unique factorization domain (UFD). A similar result holds for $6$-tuples of polynomials in one variable (see Corollary~\ref{conj-1-var}, cf.~\cite[Theorem~7.2]{Kocik-07}) because $\mathbb{H}[u]$ is still a UFD in a sense \cite[Theorem~1 in Chapter~2]{Ore}, cf. \cite{Gelfand-etal-05, Motzkin-etal,Gohberg-etal-82,Smertnig-16}, \cite[\S3.5]{Gentili-etal-13}. 
Next, $6$-tuples of degree $2$ polynomials in $2$ variables were described recently by Koll\'ar \cite[Theorems~3--8]{Kollar-16}; cf.~Remark~\ref{rem-Kollar}. He used algebraic geometry of the Veronese surface; this seems also achievable by factorization; cf. Problem~\ref{pr-deg2}. Factorization is helpless in degree $4$ because $\mathbb{H}[u,v]$ is \emph{not} a UFD, with a degree $4$ counterexample; see Example~\ref{ex-Beauregard} \cite{Beauregard-93}.
This cannot be repaired by a divisor theory because of no adequate multiplication of ideals in noncommutative setup.

Another method to describe the solution set is to give \emph{transformations} producing all solutions from a few initial ones recursively. For integral Pythagorean $n$-tuples, $n\le 9$, this was done in \cite{Cass-90}. But this does not give a parametrization of the solution set and is not easily generalized to polynomials.

Finally, a parametrization of Pythagorean $n$-tuples \emph{up to common factor} is immediately given by the inverse stereographic projection. But this does not allow to extract the required degree 4 solutions.

So the case of $6$-tuples of polynomials of degree $2$ in $u$ and $v$ arising in our geometric problem seems to be the simplest case not accessible by known methods. (Description of $5$-tuples is harder.)


\vspace{-0.3cm}

\subsection*{Plan of the proof}

The proof of Theorem~\ref{mainthm} consists of the following $3$ steps:
\begin{description}
\item[Theorem~\ref{haupt}:] reduction of finding surfaces in $S^n$ to parametrization of Pythagorean $(n+2)$-tuples;
\item[Theorem~\ref{th-pythagorean}:] parametrization of Pythagorean $6$-tuples of small degree; this gives surfaces in $S^4$;
\item[Theorem~\ref{cor-3D}:] extraction of surfaces in $\mathbb{R}^3$ from the obtained set of surfaces in $S^4$.
\end{description}

The first step is standard and goes along the lines of \cite{schicho:2001}. The second step is the heart of the proof and requires completely new ideas. The third step is a new result but the methods are less novel.

Denote by $\mathbb{R}_{mn}\subset\mathbb{R}[u,v]$ the set of polynomials of degree at most $m$ in $u$ and at most $n$ in $v$. 

\begin{thm}\label{haupt}\label{param} 
Assume that through each point of
an analytic surface in $S^{n}$
one can draw two noncospheric circular arcs fully contained in the surface (and analytically depending on the point).
Assume that through each point in a dense subset of the surface one can draw only finitely many (not nested) circular arcs fully contained in the surface.
Then the surface (besides a one-dimensional subset) has a parametrization
\begin{equation}\label{eq-circular}
\Phi(u,v)=X_{1}(u,v):\dots:X_{n+2}(u,v),
\end{equation}
where $X_{1},\dots,X_{n+2}\in\mathbb{R}_{22}$ satisfy Eq.~\eqref{eq-pythagorean-n}.
%
\end{thm}

Here it is crucial that the two circles are noncospheric. But in the opposite case, for $n=3$, it was known that such a surface projects stereographically to a Darboux cyclide \cite[Theorem~20 in p.~296]{coolidge:1916}.

It is convenient to state the next theorems in terms of quaternions; see~\cite{Lavicka-etal-07} for a basic introduction. 
Identify $\mathbb{R}^4$ with the set $\mathbb{H}$ of quaternions, and $\mathbb{R}^3$ with the set $\mathrm{Im}\mathbb{H}$ of purely imaginary quaternions.
Denote by $\mathbb{H}[u,v]$ the set of polynomials with quaternionic coefficients (the variables commute with each other and the coefficients).
Define $\mathbb{H}_{mn}$, $\mathbb{C}_{mn}$ analogously to $\mathbb{R}_{mn}$ and set
$\H_{*n}:=\bigcup_{m=1}^{\infty}\H_{mn}$.  

\begin{thm}
\label{th-pythagorean}
Polynomials $X_{1},\dots,X_6\in\mathbb{R}_{22}$ satisfy Eq.~\eqref{eq-pythagorean-n} with $n=4$
if and only if up to 
a linear transformation $\mathbb{R}^6\to\mathbb{R}^6$ preserving this equation (and not depending on the variables $u,v$) we have
\begin{equation}\label{eq-param-pythagorean}
\begin{aligned}
X_1+iX_2+jX_3+kX_4&=2ABCD,\\
X_5&=(|B|^2-|AC|^2)D,\\
X_6&=(|B|^2+|AC|^2)D
\end{aligned}
\end{equation}
for some $A,B,C\in\mathbb{H}_{11}$, $D\in\mathbb{R}_{22}$ such that $|B|^2D,|AC|^2D\in\mathbb{R}_{22}$.
\end{thm}

Beware that  representation~\eqref{eq-param-pythagorean} is in general achievable only after a linear transformation.

For a polynomial $A\in\mathbb{H}[u,v]$ and real numbers $\hat u,\hat v$ (but not quaternions) the value $A(\hat u,\hat v)\in\mathbb{H}$ is well-defined. Thus $A$ or a rational expression in such polynomials defines a (possibly degenerate)
surface in $\mathbb{R}^4$.
The stereographic projection $X_1:\dots:X_6\mapsto (X_1+iX_2+jX_3+kX_4)/(X_6-X_5)$ takes the (possibly degenerate) surface in $S^4$ given by \eqref{eq-circular} and~\eqref{eq-param-pythagorean} for $n=4$ and $ABCD\ne 0$ to
\begin{equation}\label{eq-4DBconj}
\begin{split}
&\Phi(u,v)=\bar A(u,v)^{-1}B(u,v)\bar C(u,v)^{-1},\\
&\mbox{where }A,B,C\in\mathbb{H}_{11}\mbox{ satisfy } AC\in\mathbb{H}_{11}- \{0\}.
\end{split}
\end{equation}

\begin{thm}\label{cor-3D} If an analytic surface in $\mathbb{R}^3\subset\mathbb{R}^4$ is a M\"obius transformation of surface~\eqref{eq-4DBconj} in $\mathbb{R}^4$, then the former surface is a M\"obius transformation of a subset of~(E), or~(C), or~(D).
\end{thm}

%





\subsection*{Main tools}

Let us show the origin of key Theorem~\ref{th-pythagorean}. We start with a naive approach (working only for a UFD), then extract the major obstacle, and finally introduce a tool to overcome it.

Denote $Q:=X_1+iX_2+jX_3+kX_4\in \mathbb{H}[u,v]$, $P:=X_6-X_5$, $R:=X_6+X_5$. For $n=4$ Eq.~\eqref{eq-pythagorean-n}  and~\eqref{eq-param-pythagorean} are equivalent to $Q\bar Q=PR$ and $(P,Q,R)=(2|AC|^2D,2ABCD,2|B|^2D)$ respectively.

The equation $Q\bar Q=PR$ is easily solved in the commutative UFD $\mathbb{C}[u,v]$.
All solutions with $\bar P=P,\bar R=R$ are parametrized by $(P,Q,R)=(A\bar AD,ABD,B\bar BD)$, where $A,B,D\in\mathbb{C}[u,v]$ satisfy $\bar D=D$; cf.~\cite[Proof of Theorem~2.2]{Dietz-etal}. In other words, after cancellation of a common divisor, $Q$ splits into a product of two factors of norm squares $P$ and $R$.

Neither this assertion nor unique factorization hold in $\mathbb{H}[u,v]$ in any reasonable sense:


\begin{exmp} \label{ex-Beauregard} (Beauregard \cite{Beauregard-93})
The polynomial
$Q=u^2v^2-1 + (u^2- v^2)i + 2uvj$ is irreducible in $\mathbb{H}[u,v]$ but $Q\bar Q=(u^2-\sqrt{2}u+1)(v^2-\sqrt{2}v+1)\cdot(u^2+\sqrt{2}u+1)(v^2+\sqrt{2}v+1)=:P\cdot R$. Thus
\begin{equation*}
  Q\bar Q=(u-e^{i\pi/4})(u-e^{3i\pi/4})(u-e^{5i\pi/4})(u-e^{7i\pi/4})(v-e^{i\pi/4})(v-e^{3i\pi/4})(v-e^{5i\pi/4})(v-e^{7i\pi/4})
\end{equation*}
are two decompositions in $\mathbb{H}[u,v]$ with different number of irreducible factors.
\end{exmp}

We thus have to solve a nonlinear equation over a noncommutative ring which is not a UFD. We could not find any known methods suitable for that and have to develop a completely new approach.

To overcome the obstacle, we perform linear transformations preserving the equation $Q\bar Q=PR$ and simplifying the solution in a sense. 
Our transformations usually have form
\begin{equation} \label{eq-linear-transformation}
(P',\,Q',\,R')=(P-T\bar Q-Q\bar T+TR\bar T,\,Q-TR,\,R)
\end{equation}
with $T\in\mathbb{H}$.
They correspond to linear transformations $\mathbb{R}^6\to\mathbb{R}^6$ in Theorem~\ref{th-pythagorean}. Transformations~\eqref{eq-linear-transformation} with nonconstant $T\in\mathbb{H}[u,v]$ are also instrumental, although not allowed in the theorem. An interlacing of factorization and transformation methods leads to a surprisingly short proof~of~\ref{th-pythagorean}. 


\begin{exmp} \label{ex-transformation}
For $T=j$ the polynomials from  Example~\ref{ex-Beauregard} are transformed to 
$(P',Q',R')=(|AC|^2,ABC,|B|^2)$,
where
$A=(1-j)\left(u+\frac{-i-j}{\sqrt2}\right)$,
$B=\left(v+\frac{1+k}{\sqrt2}\right)
\left(u+\frac{1+i}{\sqrt2}\right)$,
$C=v+\frac{-j-k}{\sqrt2}$.
\end{exmp}

\begin{rem}
Almost every $T,A,B,C$ of the same degrees as in ~\ref{ex-transformation} produce examples like~\ref{ex-Beauregard}.
\end{rem}

\subsection*{Organization of the paper}

Theorems~\ref{mainthm}, \ref{haupt}, \ref{th-pythagorean}, \ref{cor-3D} are proved in \S\S~\ref{sec-open}, \ref{sec-parametrization}, \ref{sec-factorization}, \ref{sec-classification} respectively. Sections~\ref{sec-factorization}--\ref{sec-parametrization} are independent except that Splitting Lemma~\ref{l-splitting-basic} from \S\ref{sec-factorization} is also used in \S\ref{sec-classification}. Some open problems are stated in Section~\ref{sec-open}.

The paper is written in a mathematical level of rigor, meaning that all the definitions, conventions, theorems, corollaries, and lemmas should be understood literally. Theorems and lemmas remain true, even if cut out from the text, under the global conventions made at the beginning of each section. The proofs of lemmas use the statements but not the proofs of the other lemmas. Straightforward proofs of examples and corollaries are usually omitted because they are not used in the other proofs.

We tried our best to make the proof accessible to nonspecialists. For used known results, we tried to find references to the proofs written for nonspecialists as well, with no essential details omitted. When unable to find such ones, we preferred to include the proofs in the paper.

\section{Factorization results}
\label{sec-factorization}

In this section we prove results on factorization of quaternionic polynomials including Theorem~\ref{th-pythagorean} and Example~\ref{ex-Beauregard}.
We concentrate on the ones required for the main theorem; factorization of polynomials in one variable is studied in detail in \cite{Gentili-etal-13, Gentili-Stoppato-08}. The proofs are self-contained. We need several lemmas, in which the equation $Q\bar Q=PR$ is solved step by step for $Q$ of degree $0$ or $1$ in $v$. Examples below show that the degree bounds in the lemmas are essential. Assertions~\ref{l-preserve}--\ref{ex-first} do not pretend to be new, although we did not find them in the literature.



\begin{lem} \label{l-preserve} Transformation~\eqref{eq-linear-transformation} preserves the equation
$Q\overline{Q}=PR$, i.e., for each $T,Q\in\mathbb{H}[u,v]$ and $P,R\in\mathbb{R}[u,v]$ such that $Q\overline{Q}=PR$ we have $Q'\overline{Q}{}'=P'R'$.
\end{lem}

\begin{proof} We have
  $Q'\overline{Q}{}'=(Q-TR)(\overline{Q}-\overline{T}R)=
  PR-T\overline{Q}R-Q\overline{T}R+TR\overline{T}R=P'R'$.
\end{proof}

\begin{lem} \label{l-adding variable2} If polynomials $Q\in\mathbb{H}[u]$ and $P,R\in\mathbb{R}[u]$ satisfy $Q\bar Q=PR$, then $(P,Q,R)=(A\bar AD,ABD,B\bar BD)$ for some $A,B\in\mathbb{H}[u]$ and $D\in\mathbb{R}[u]$.
\end{lem}


\begin{proof} Use induction over $\deg Q$. The base is $Q=0$, that is, $\deg Q=-\infty$. Then $P$ or $R$, say, $R$ vanishes. Clearly, $A=1$, $B=0$, $D=P$ are the required polynomials.

To make induction step, assume that $Q\ne 0$ and the lemma holds for all $Q$ of smaller degree. We have $P,R\ne 0$. Thus $P$ or $R$ has degree at most $\deg Q$. Assume that $\deg R\le \deg Q$; otherwise apply the same argument to the triple $(R,\overline{Q}, P)$ instead of $(P,Q,R)$. Divide each of the $4$ components of $Q$ by $R$ with remainders in $\mathbb{R}[u]$. We get $Q=TR+Q'$ for some $T,Q'\in\mathbb{H}[u]$ and $\deg Q'<\deg R\le \deg Q$. Perform transformation~\eqref{eq-linear-transformation}. By Lemma~\ref{l-preserve} it preserves the equation $Q\bar Q=PR$ and decreases $\deg Q$. By the inductive hypothesis, $Q'=A'BD$ and $R=B\bar BD$ for some $A',B\in\mathbb{H}[u]$ and some $D\in\mathbb{R}[u]$. Thus $Q=Q'+TR=(A'+ T\bar B)BD=:ABD$ and $P=Q\bar Q/R=A\bar AD$ as required.
\end{proof}

\begin{rem} Alternatively, Lemma~\ref{l-adding variable2} can be deduced from Quaternionic Fundamental Theorem of Algebra \cite[Theorem~5.2]{Gentili-Stoppato-08}: \emph{each polynomial in $\mathbb{H}[u]$ admits a factorization into degree $1$ factors}; cf.~\cite{Eilenberg-etal-44,Gentili-etal-13, Motzkin-etal}. Here the possibility of ``interchanging factors'' is important \cite[Theorem~5.4]{Gentili-Stoppato-08}.
\end{rem}

\begin{cor} \label{conj-1-var} Polynomials $X_{1},\dots,X_6\!\in\!\mathbb{R}[u]$ satisfy~\eqref{eq-pythagorean-n} for $n\!=\!4$, if and only if for some
$A,\!B\!\in\!\mathbb{H}[u]$, $D\in\mathbb{R}[u]$ we have
$X_1+iX_2+jX_3+kX_4=2ABD,\,
X_5=(|B|^2-|A|^2)D,\,
X_6=(|B|^2+|A|^2)D.
$
\end{cor}

\begin{exmp} \label{ex-first}
For $A=u+i$, $B=v+j$ there are no $A',B',D\in\mathbb{H}[u,v]$ such that $(A\bar A, BA,B\bar B)=(A'\bar A{}'D,A'B'D,B'\bar B{}'D)$.
\end{exmp}

\begin{splitting-lemma}\label{l-splitting-basic}
If polynomials $Q\in\mathbb{H}_{11}$, $P\in\mathbb{R}_{02}$, $R\in\mathbb{R}_{20}$ satisfy $|Q(u,v)|^2=P(v)R(u)$
then $Q(u,v)=A(v)B(u)$ or $Q(u,v)=B(u)A(v)$ for some $A\in\mathbb{H}_{01}$, $B\in\mathbb{H}_{10}$.
\end{splitting-lemma}

\begin{proof}
%
%
%
Assume that $\deg P=\deg R=2$; otherwise $Q$ does not depend on one of the variables and there is nothing to prove. Expand $Q(u,v)=:Q_0(u)+vQ_1(u)=:Q_{00}+Q_{10}u+Q_{01}v+Q_{11}uv$. 
Since $\deg P=\deg R=2$, we have $Q_{11}\ne 0$. Take $q\in\mathbb{H}$ such that $Q_0(u)+qQ_1(u)$ is a constant and denote the constant by $p$. That is, set
$q:=-Q_{10}Q_{11}^{-1}$ and $p:=Q_0+qQ_1=Q_{00}-Q_{10}Q_{11}^{-1}Q_{01}$.

Consider the polynomial $|Q|^2(u,q)$ obtained by substitution of the quaternion $q$ into the \emph{real} polynomial $|Q|^2(u,v)$.
(The resulting polynomial is \emph{not} necessarily real and should not be confused with the real polynomial $|Q(u,q)|^2$ not used and not defined in the paper.) On one hand, $|Q|^2(u,q)=P(q)R(u)$ is divisible by $R(u)$ of degree $2$. On the other hand,
$|Q|^2(u,q)=q(qQ_1+Q_0)\bar Q_1+(qQ_1+Q_0)\bar Q_0=qp\bar Q_1+p\bar Q_0$ has degree at most $1$.
Thus $|Q|^2(u,q)=qp\bar Q_1+p\bar Q_0=0$ identically.

Now for $p\ne 0$ we get $Q_0=-Q_1\bar p\,\bar q\,\bar p^{-1}$, hence $Q=Q_1(u)(v-\bar p\,\bar q\,\bar p^{-1})$ as required.
For $p=0$ we get $Q_0=-qQ_1$ by the definition of $p$, hence $Q=Q_0+vQ_1=(v-q)Q_1(u)$ as required. 
\end{proof}


\begin{lem} \label{l-splitting4}
If polynomials $Q\in\mathbb{H}_{*1}$, $P\in\mathbb{R}_{*2}$, $R\in \mathbb{R}_{20}$ satisfy $Q\bar Q=PR$, then $(P,Q,R)$ or $(R,Q,P)$ equals $(A\bar AD,ABD,B\bar BD)$ for some $A,B\in\mathbb{H}[u,v]$ and $D\in\mathbb{R}[u,v]$.
\end{lem}

\begin{proof} Assume that $R\ne 0$; otherwise set $(A,B,D)=(1,0,P)$. Expand $Q(u,v)=:Q_0(u)+vQ_1(u)$ and divide each of the $4$ components of both $Q_0(u)$ and $Q_1(u)$ by $R\in\mathbb{R}[u]$ with remainders in $\mathbb{R}[u]$. We get $Q=TR+Q'$, where $T,Q'\in\mathbb{H}_{*1}$ and the degree of $Q'$ in $u$ is less than $\deg R$. Perform transformation~\eqref{eq-linear-transformation}. We get $P'\in\mathbb{R}_{02}$ because $P'R=Q'\bar Q{}'$ by Lemma~\ref{l-preserve} and the degree of $Q'$ in $u$ is less than $\deg R\le 2$. Assume that $Q'\ne0$; otherwise set $(A,B,D)=(T,1,R)$.

By Splitting Lemma~\ref{l-splitting-basic} it follows that $Q'$ is a product of two factors $A'\in\mathbb{H}_{01}$ and $B\in\mathbb{H}_{10}$ in some order. Set $D:=\frac{R(u)}{|B(u)|^2}=\frac{|A'(v)|^2}{P'(v)}\ne 0$. We get  $R=B\bar BD$, $D\in\mathbb{R}(u)\cap \mathbb{R}(v)=\mathbb{R}$, and
$$
Q=TR+Q'=\begin{cases}(D^{-1}A'+T\bar B)BD=:ABD, & \text{if } Q'=A'B;\\
B(A'D^{-1}+\bar BT)D=:BAD, & \text{if } Q'=BA'.\end{cases}
$$
Then $P=Q\bar Q/R=A\bar AD$. In the case when $Q=BAD$ it remains to relabel $A$ and $B$.
\end{proof}

\begin{exmp} \label{ex-ABC} For $(A,B,C)=(u+i,v+j,u+k)$ or $(u+i,uv+j,v+k)$
there are no $A',B',D\in\mathbb{H}[u,v]$ such that $\{|AC|^2,ABC,|B|^2\}=\{A'\bar A{}'D,A'B'D,B'\bar B{}'D\}$.
\end{exmp}

\begin{lem}\label{th-linear-PDT}
If polynomials $Q\in\mathbb{H}_{*1}$ and $P,R\in\mathbb{R}[u,v]$ satisfy $Q\bar Q=PR$, then $(P,Q,R)$ or $(R,Q,P)$ equals $(|AC|^2D,ABCD,|B|^2D)$ for some $A,C\in\mathbb{H}[u]$, $B\in\mathbb{H}[u,v]$, $D\in\mathbb{R}[u,v]$.
\end{lem}

\begin{proof}
Step 1. Let us reduce the lemma to the particular case when $R\in \mathbb{R}[u]$, $P\not\in\mathbb{R}[u]$, and $Q, R$ have no common nonconstant real divisors.

Assume that $R\ne 0$; otherwise set $(A,B,C,D)=(1,0,1,P)$.

First assume that $Q$ and $R$  have a common real divisor $D$ irreducible in $\mathbb{R}[u,v]$. Then $PR=Q\bar Q$ is divisible by $D^2$. Hence $P$ is divisible by $D$ or $R$ is divisible by $D^2$. Replacing $(P,Q,R)$ by $(P,Q,R)/D$ or $(P,Q/D,R/D^2)$ one cancels $D$. Now if $D\in\mathbb{R}[u]$ then one applies the proof
of Lemma~\ref{th-linear-PDT} for the resulting triple
to get a factorization of $Q/D$. Multiplying an appropriate factor by $D$,
one gets the required factorization of $Q$. If $D\not\in\mathbb{R}[u]$
then the lemma follows from Lemma~\ref{l-adding variable2}. Assume further that $Q$ and $R$ have no  common nonconstant real divisors.

Now assume that $R$ has degree $1$ in $v$. Then it has a real divisor $D$ of degree $1$ in $v$ irreducible in $\mathbb{R}[u,v]$. Take any $\hat u\in\mathbb{R}$ such that $D(\hat u,v)$ has degree $1$ in $v$. The equation $D(\hat u,v)=0$ has  a real root $v(\hat u)$. Write $Q=X_1+iX_2+jX_3+kX_4$ with $X_1,X_2,X_3,X_4\in\mathbb{R}[u,v]$. We have $|Q(\hat u,v(\hat u))|^2=0$, hence $X_l(\hat u,v(\hat u))=0$ for 
$l=1,\dots,4$. By the Bezout theorem the two curves $X_l(u,v)=0$ and $D(u,v)=0$ have a common component. Since $D$ is irreducible, it divides $X_l$. Thus $D$ is a common real divisor for $R$ and $Q$, a contradiction. Hence $R$ has even degree in $v$.

Assume that $Q$ has degree $1$ in $v$; otherwise the lemma follows from Lemma~\ref{l-adding variable2}. Then without loss of generality $R\in\mathbb{R}[u]$ and $P\not\in\mathbb{R}[u]$ because $P$ and $R$ have in fact been equitable so far.

Step 2. Let us prove the lemma in the particular case stated in Step 1 by induction over $\deg R$.

The base $\deg R\le 2$ is Lemma~\ref{l-splitting4}. Assume that $\deg R>2$. Factorize $R=R'R''$ with irreducible $R'\in\mathbb{R}[u]$ so that $\deg R'$ is $1$ or $2$. Apply Lemma~\ref{l-splitting4} to the triple $(PR'',Q,R')$. We get $R'=Q'\bar Q{}'D'$, $PR''=Q''\bar Q{}''D'$, and $Q=Q'Q''D'$ or $Q=Q''Q'D'$
for some $Q',Q''\in\mathbb{H}[u,v]$, $D'\in\mathbb{R}[u,v]$. Here $D'=\mathrm{const}$ as a common divisor of $Q$ and $R$. Assume that $D'=1$; otherwise divide $P,Q,R$ by $D'$. We have $Q''\in\mathbb{H}_{*1}$ and $Q'\in\mathbb{H}[u]$ as divisors of $Q$ and $R'$ respectively (clearly, $Q\ne 0$).

Apply the inductive hypothesis to the triple $(P,Q'',R'')$. We get
$Q''=ABCD$, and also $P=|AC|^2D$ or $P=|B|^2D$
for some $A,C\in\mathbb{H}[u]$, $B\in\mathbb{H}[u,v]$, $D\in\mathbb{R}[u,v]$. Here again $D=\mathrm{const}$. 
Hence $|AC|^2D\in\mathbb{R}[u]$.
Thus $P=|B|^2D$ because $P\not\in\mathbb{R}[u]$. Further,
$Q=Q'Q''=(Q'A)BCD$ or $Q=Q''Q'=AB(CQ')D$. Finally,
$R=Q\bar Q/P=|(Q'A)C|^2D=|A(CQ')|^2D$,
as required.
\end{proof}

\begin{cor} \label{cor-reducibility} (Cf.~\cite[Proposition~3]{Beauregard-93})
The norm square of a polynomial $Q\in\mathbb{H}_{*1}$ is reducible in $\mathbb{R}[u,v]$ if and only if $Q$ is reducible in $\mathbb{H}[u,v]$ or equals a real polynomial times a constant quaternion. 
\end{cor}




Example~\ref{ex-Beauregard} shows that Lemma~\ref{th-linear-PDT} and Corollary~\ref{cor-reducibility} do not remain true
for~$Q\in \mathbb{H}_{*2}$.

\begin{proof}[Proof of Example~\ref{ex-Beauregard}] Let us prove that $Q$ is irreducible in $\mathbb{H}[u,v]$ (in \cite{Beauregard-93} the irreducibility in the ring of \emph{rational} quaternions is proved). Assume that $Q=AB$, where $A,B\in\mathbb{H}[u,v]$ are not constant.

First consider the case when one of these polynomials, say, $A$, does not depend on one of the variables, say, $v$. Then write $Q=(u^2-i)v^2+(2ju)v+(iu^2-1)$. Both $u^2-i$ and $2ju$ must be left-divisible by $A\ne\mathrm{const}$, which leads to a contradiction by taking $u=0$.

Since $Q\in\mathbb{H}_{22}$ it remains to consider the case when $A,B\in\mathbb{H}_{11}$. The polynomial 
$|Q|^2=|A|^2|B|^2$, and hence its divisor $|A|^2$, is a product of real polynomials in one variable. 
By Splitting Lemma~\ref{l-splitting-basic} the polynomial $A$ is reducible in $\mathbb{H}[u,v]$. The left factor in a decomposition of $A\in\mathbb{H}_{11}$ is a left divisor of $Q$ and does not depend on one of the variables, which leads to the case already considered above.
In both cases we get a contradiction which proves that $Q$ is irreducible.
\end{proof}

\begin{exmp}
The irreducible polynomial $Q$ from Example~\ref{ex-Beauregard} satisfies another surprising identity:
$(u^2+1)Q=
\left(u+\frac{-k-j}{\sqrt2}\right)
\left(v+\frac{1-i}{\sqrt2}\right)
\left(u+\frac{1+k}{\sqrt2}\right)
\left(u+\frac{-1+k}{\sqrt2}\right)
\left(v+\frac{-1+i}{\sqrt2}\right)
\left(u+\frac{j-k}{\sqrt2}\right).
$
\end{exmp}

\begin{proof}[Proof of Theorem~\ref{th-pythagorean}]
The `if' part is straightforward.
Let us prove the `only if' part. 

The initial polynomials $X_1,\dots X_6$ do not admit the desired representation~\eqref{eq-param-pythagorean} in general. By means of linear transformations of $(X_1,\dots X_6)$ we are going to reduce the theorem to a case in which they do. Recall that $Q:=X_1+iX_2+jX_3+kX_4$, $P:=X_6-X_5$, $R:=X_6+X_5$.

Assume that $P,Q,R$ have degree $2$ in each of the variables and have no common nonconstant real divisors; otherwise the theorem follows from Lemma~\ref{th-linear-PDT}.
Expand $Q(u,v)=:Q_0(u)+Q_1(u)v+Q_2(u)v^2$. Define $P_2(u)$ and $R_2(u)$ analogously. Since $Q\bar Q=PR$ it follows that  $Q_2\bar Q_2=P_2R_2$. If $\deg Q_2=2$ then $\deg P_2=\deg R_2=2$ and transformation~\eqref{eq-linear-transformation} for appropriate $T\in\mathbb{H}$ kills the leading term of $Q_2$. This is a linear transformation $\mathbb{R}^6\to \mathbb{R}^6$ preserving equation~\eqref{eq-pythagorean-n} by Lemma~\ref{l-preserve}. Thus we may assume that $\deg Q_2\le 1$. Then $\deg P_2\le 1$ or $\deg R_2\le 1$. Assume that $\deg R_2\le 1$; otherwise interchange $P$ and $R$, which is one more linear transformation.

Since $Q_2\bar Q_2=P_2R_2$ and $\deg R_2\le 1$, by Lemma~\ref{l-adding variable2} it follows that $Q_2=TR_2$ for some $T\in\mathbb{H}_{10}$.
Now transformation~\eqref{eq-linear-transformation}  kills $Q_2$, but this time it does not correspond to a linear map $\mathbb{R}^6\to\mathbb{R}^6$ because $T$ may depend on $u$. We get
$Q'=Q-TR\in\mathbb{H}_{31}$. 
By Lemma~\ref{th-linear-PDT} we have $Q'=ABC'D$, and also $R=|B|^2D$ or $R=|AC'|^2D$ for some $A,C'\in\mathbb{H}[u]$, $B\in\mathbb{H}[u,v]$, $D\in\mathbb{R}[u,v]$.

Here $D=\mathrm{const}$ as a common real divisor of $P,Q,R$. Hence $|AC'|^2D\in\mathbb{R}[u]$, thus $R=|B|^2D$ because $R\not\in\mathbb{R}[u]$.  We have $ABC'D=Q'\ne 0$; otherwise $R$ is a common divisor of $P,Q,R$. Since $ABC'D\in\mathbb{H}_{31}$
it follows that $A$ or $C'$, say, $A$ has degree at most $1$. Divide $T\in\mathbb{H}_{10}$ by $A\in\mathbb{H}_{10}$ from the left with a remainder: $T=AS+T'$, where $S\in\mathbb{H}[u]$ and $T'\in\mathbb{H}$. 
We get
$$Q=Q'+TR=ABC'D+T|B|^2D=AB(C'+\bar BS)D+T'|B|^2D=:ABCD+T'|B|^2D.$$
So transformation~\eqref{eq-linear-transformation} with $T$ replaced by $T'\in\mathbb{H}$ takes $(P,Q,R)$ to $(|AC|^2D,ABCD,|B|^2D)$.

Since $P,Q,R\in\mathbb{H}_{22}-\{0\}$, the required degree restrictions on $A,B,C,D$ in Eq.~\eqref{eq-param-pythagorean} follow.
\end{proof}



\section{Classification results}\label{sec:proofs}
\label{sec-classification}

In this section we prove the results on the classification of surfaces parametrized by quaternionic rational functions of small degree up to M\"obius transformations, in particular, Theorem~\ref{cor-3D}. We identify $\mathbb{R}^4$ with $\mathbb{H}$, $\mathbb{R}^3$ with $\mathrm{Im}\mathbb{H}$, and use basic geometric properties of quaternions; see e.g.~\cite{Lavicka-etal-07}. Throughout this section by a \emph{(hyper)surface} 
we mean a real analytic map of a domain in $\mathbb{R}^m$ to $\mathbb{R}^n$ not necessarily having nondegenerate differential nor being injective, 
and also the image of the map, 
if no confusion arises. Thus a (hyper)surface can degenerate to a set of dimension less than $m$.
We prefer to perform no complexification because quaternionic polynomials do not have well-defined values at complex points. Theorem~\ref{cor-3D} is deduced from the following one.

\begin{thm}\label{prop21}
\label{prop:21-split}
If surface~\eqref{eq-4DBconj} 
is contained in $\mathbb{R}^3$ (respectively, in $S^3$) then it is a subset of a
Euclidean (respectively, Clifford) translational surface or
a Darboux cyclide (respectively, an intersection of $S^3$ with another $3$-dimensional quadric in $\mathbb{R}^4$).
\end{thm}

\begin{exmp} The surface $\Phi(u,v)=(u-i)((2j+i)v-2i-j)\left((u-i)(v-k)\right)^{-1}$ is a torus in $\mathbb{R}^3$ and
$\Phi(u,v)=(v+i)^{-1}(v+j)(u+k)(u+i)^{-1}$ is a Clifford translational surface in $S^3$.
\end{exmp}

Let us prove Theorem~\ref{prop21}.
We start with a folklore lemma, which in particular implies that almost every surface~\eqref{eq-4DBconj} contains $2$ circles $u=\mathrm{const}$ and $v=\mathrm{const}$ through almost every point.

\begin{lemma}\label{l-axial} The curve $\gamma(u)=(au+b)(cu+d)^{-1}\subset\mathbb{H}$, where $a,b,c,d\in\mathbb{H}$, $c\ne 0$, $b-ac^{-1}d\ne 0$, $dc^{-1}\not\in\mathbb{R}$, are fixed and $u$ runs through $\mathbb{R}$, is a circle with one point missing. If $\gamma(u)$ is contained in $\im\H$, then the plane of $\gamma(u)$ is orthogonal to the vector $\im\, (dc^{-1})$.
The same is true for the curve $\gamma(u)=(cu+d)^{-1}(au+b)$, where $c\ne 0$, $b-dc^{-1}a\ne0$, $c^{-1}d\notin\mathbb{R}$, and the vector~$\mathrm{Im}(c^{-1}d)$.
\end{lemma}

\begin{proof} We have
$\gamma(u)=ac^{-1}+(b-ac^{-1}d)(cu+d)^{-1}=:f+g(u+h)^{-1}$ for some $f,g,h\in\mathbb{H}$, where $h=dc^{-1}$.
This is a composition of translations, rotation, homothety, and the unit inversion $q\mapsto \bar q^{-1}$ applied to the line $\mathbb{R}\subset \mathbb{H}$, once $b-ac^{-1}d\ne 0$. Since $h\not\in\mathbb{R}$, the inversion is applied to a line not passing through the origin. Thus $\gamma(u)$ is a circle with one point missing.

Now assume that $\gamma(u)\subset\im\H$. Then $0=\re\gamma(u)|u+h|^2=\re f|u+h|^2+\re g u+\re (g\bar h)$. Since the coefficients of the polynomial in the right-hand side vanish it follows that
$f,g\in\im\H$ and
$0=\mathrm{Re}(g\bar h)=\mathrm{Re}(g\mathrm{Re}h+\langle g,\mathrm{Im}h\rangle-g\times\mathrm{Im}h)
=\langle g,\mathrm{Im}h\rangle$, where $\langle \cdot ,\cdot\rangle$ is the scalar product in $\mathbb{R}^3$. Hence
$g$ is orthogonal to $\im h$. 
The circle $g(u+h)^{-1}=(gu+g\mathrm{Re}h-g\times\mathrm{Im}h)|u+h|^{-2}$ is contained in the plane spanned by $g$ and $g\times \im h$. So $\gamma(u)=f+g(u+h)^{-1}$ is contained in a plane orthogonal to $\im h$ as well.
For $\gamma(u)=(cu+d)^{-1}(au+b)$ the proof is analogous.
\end{proof}

Now consider a particular case of surface~\eqref{eq-4DBconj} given by $\Phi(u,v)=A(u,v)B(u,v)^{-1}$, 
where
$$A=A_{11}uv+A_{10}u+A_{01}v+A_{00}\in\mathbb{H}_{11}-\{0\},\quad B=B_{11}uv+B_{10}u+B_{01}v+B_{00}\in\mathbb{H}_{11}-\{0\}.$$
We are going to estimate its degree.
For that we estimate the degree of the hypersurface 
\begin{equation}\label{eq-hyper}
\tilde\Phi(u,v,w,s)=\tilde A(u,v,w,s)\tilde B(u,v,w,s)^{-1},
\end{equation}
where the linear homogeneous polynomials $\tilde A,\tilde B\in\mathbb{H}[u,v,w,s]-\{0\}$ are given by the formulae
$$\tilde A=A_{11}w+A_{10}u+A_{01}v+A_{00}s, \quad \tilde B=B_{11}w+B_{10}u+B_{01}v+B_{00}s.$$
Thus $A(u,v)=\tilde A(u,v,uv,1)$ and $B(u,v)=\tilde B(u,v,uv,1)$. 
The hypersurface 
has the rational parametrization $\tilde\Phi=\tilde A\overline{\tilde B}/|\tilde B|^2$ of degree at most $2$ in each variable, and by elimination of variables (not used in the paper)
it has degree at most $8$. We prove a much sharper estimate.

\begin{lem}\label{l-3folddeg4} 
Hypersurface~\eqref{eq-hyper} is contained in an algebraic hypersurface of degree at most $4$.
\end{lem}

\begin{proof} A point $t+ix+jy+kz\in\mathbb{H}$ is contained in the hypersurface $\tilde A\tilde B^{-1}$ 
if and only if there are
$u,v,w,s\in\mathbb{R}$ not vanishing simultaneously such that $\tilde A(u,v,w,s)-(t+ix+jy+kz)\cdot\tilde B(u,v,w,s)=0$ and $\tilde B(u,v,w,s)\ne 0$. The latter quaternionic equation can be considered as a system of $4$ real linear homogeneous equations in the variables $u,v,w,s$ with the coefficients linearly depending on the parameters $x,y,z,t$. The system has a nonzero solution if and only if the determinant of the system vanishes, which gives an algebraic equation in $x,y,z,t$ of degree at most $4$. The algebraic equation defines the required algebraic hypersurface unless the determinant vanishes identically. The latter situation is indeed possible, e.g., if $\tilde A$ and $\tilde B$ do not depend on $s$.

Assume that the determinant vanishes identically. Then the system has a nonzero solution for each point $(x,y,z,t)\in\mathbb{R}^{4}$. But the set of values of the fraction $\tilde A\tilde B^{-1}$ is at most three-dimensional. One can have a nonzero solution for each point in $\mathbb{R}^{4}$ only if $\tilde A(\hat u,\hat v,\hat w,\hat s)=\tilde B(\hat u,\hat v,\hat w,\hat s)=0$ for some nonzero $(\hat u,\hat v,\hat w,\hat s)\in\mathbb{R}^{4}$. Assume that $\hat s\ne 0$ without loss of generality. Then by the linearity
\begin{align*}
\tilde A(u,v,w,s)&=\hat s\tilde A(u,v,w,s)/\hat s-s\tilde A(\hat u,\hat v,\hat w,\hat s)/\hat s=\tilde A(\hat s u-\hat us, \hat s v-\hat vs, \hat s w-\hat ws,0)/\hat s\\
\tilde B(u,v,w,s)&=\hat s\tilde B(u,v,w,s)/\hat s-s\tilde B(\hat u,\hat v,\hat w,\hat s)/\hat s=\tilde B(\hat s u-\hat us, \hat s v-\hat vs, \hat s w-\hat ws,0)/\hat s.
\end{align*}
Performing a linear change of the parameters $u,v,w,s$ we may assume that both $\tilde A(u,v,w,s)$ and $\tilde B(u,v,w,s)$ do not depend on $s$. Further denote by $\tilde A(u,v,w)$ and $\tilde B(u,v,w)$ the resulting linear homogeneous polynomials, defining the same hypersurface $\tilde A\tilde B^{-1}$ 
as the initial ones.

Then the above equation takes the form $\tilde A(u,v,w)-(t+ix+jy+kz)\cdot\tilde B(u,v,w)=0$. Consider the equation as a system of $4$ real linear homogeneous equations in the variables $u,v,w$. The system has a nonzero solution if and only if all the $3\times 3$ minors of the system vanish, which gives four algebraic equations in $x,y,z,t$ of degree at most $3$. If at least one of the $3\times 3$ minors does not vanish identically, then it defines
the required algebraic surface. If all the $3\times 3$ minors vanish identically then repeat the argument of the previous paragraph to get
linear homogeneous $\tilde A(u,v)$ and $\tilde B(u,v)$ depending only on $2$ variables. Again, we either get the required algebraic surface or proceed to the case when $\tilde A(u)$ and $\tilde B(u)$ depend only on $1$ variable. In the latter case any hyperplane passing through the point $\tilde A(1)\tilde B(1)^{-1}$ is the required algebraic hypersurface.
\end{proof}

We consider the case when the constructed algebraic hypersurface degenerates to $\mathrm{Im}\mathbb{H}$ separately. This is indeed possible, e.g., if $\tilde A\in \mathrm{Im}\mathbb{H}[u,v,w,s]$, $\tilde B\in\mathbb{R}[u,v,w,s]$ or $\tilde A=\tilde B\cdot d$ for some $d\in\mathrm{Im}\mathbb{H}$. Hereafter $\mathrm{Im}\mathbb{H}[u,v,w,s]$ denotes the set (not a ring) of polynomials with the coefficients in $\mathrm{Im}\mathbb{H}$.

\begin{lem}\label{l-3fold-R3} 
Assume that hypersurface~\eqref{eq-hyper} 
is contained in $\mathrm{Im}\mathbb{H}$. Then the map $\tilde A\tilde B^{-1}$ is a composition of a map of the form either $\tilde C\tilde D\tilde C^{-1}$ or $\tilde E\tilde F^{-1}$ with a M\"obius transformation of $\mathrm{Im}\mathbb{H}$,
where $\tilde C\in\mathbb{H}[u,v,w,s]$, $\tilde E\in\mathrm{Im}\mathbb{H}[u,v,w,s]$, $\tilde F\in\mathbb{R}[u,v,w,s]$ are linear homogeneous, and $\tilde D\in\mathrm{Im}\mathbb{H}$.
\end{lem}

\begin{proof} The inverse stereographic projection
$
\im\H \to S^3, q \mapsto (1+q)(1-q)^{-1},
$
maps the hypersurface $\tilde A\tilde B^{-1}$ contained in $\im\mathbb{H}$ to the hypersurface $(\tilde A+\tilde B)(\tilde B-\tilde A)^{-1}$ contained in $S^3$. Thus $|\tilde A+\tilde B|=|\tilde B-\tilde A|$ identically. In particular, for each $(u,v,w,s)\in\mathbb{R}^4$ the condition $\tilde B(u,v,w,s)-\tilde A(u,v,w,s)=0$ implies the condition $\tilde B(u,v,w,s)+\tilde A (u,v,w,s)=0$.

Denote $\tilde B(u,v,w,s)-\tilde A(u,v,w,s)=:a_1u+a_2v+a_3w+a_4s$
and $\tilde B(u,v,w,s)+\tilde A(u,v,w,s)=:b_1u+b_2v+b_3w+b_4s$. Define a real linear map from the real linear span of $a_1,a_2,a_3,a_4\in\mathbb{H}$ into the real linear span of $b_1,b_2,b_3,b_4\in\mathbb{H}$ by the formula
$a_1u+a_2v+a_3w+a_4s\mapsto b_1u+b_2v+b_3w+b_4s$. The map is well-defined because the condition $a_1u+a_2v+a_3w+a_4s=0$ implies $b_1u+b_2v+b_3w+b_4s=0$. The map is isometric because $|\tilde B+\tilde A|=|\tilde B-\tilde A|$. Extend it to an isometry $\mathbb{H}\to\mathbb{H}$. Each isometry $\mathbb{H}\to\mathbb{H}$ has one of the forms $q\mapsto cqd$ or $q\mapsto c\bar qd$ for some $c,d\in S^3$ \cite[Theorem~3.2]{Lavicka-etal-07}. Therefore $\tilde B+\tilde A=c(\tilde B-\tilde A)d$ or $\tilde B+\tilde A=c\overline{(\tilde B-\tilde A)}d$.

In the former case set $\tilde C:=\tilde B-\tilde A$, $\tilde D:=(1+d)(1-d)^{-1}\in\mathrm{Im}\mathbb{H}$. We have $(\tilde B+\tilde A)(\tilde B-\tilde A)^{-1}=c \tilde C d \tilde C^{-1}$. Up to a M\"obius transformation this is $\tilde C d \tilde C^{-1}$, which projects stereographically to $\tilde C \tilde D \tilde C^{-1}$.

In the latter case set $\tilde E:=\mathrm{Im}(\tilde B\bar d-\tilde A\bar d)$ and $\tilde F:=\mathrm{Re}(\tilde A\bar d-\tilde B\bar d)$. We have $(\tilde B+\tilde A)(\tilde B-\tilde A)^{-1}=c \bar d (\tilde F+\tilde E)(\tilde F-\tilde E)^{-1}$. Up to a M\"obius transformation this is $(\tilde F+\tilde E)(\tilde F-\tilde E)^{-1}$, which projects stereographically to $\tilde E \tilde F^{-1}$.
\end{proof}

\begin{lem}\label{l-quadrics} The surfaces $C D C^{-1}$ and $ E F^{-1}$, 
where $ C\in\mathbb{H}_{11}-\{0\}$, $E\in\mathrm{Im}\mathbb{H}_{11}$, $F\in\mathbb{R}_{11}-\{0\}$, $D\in\mathrm{Im}\mathbb{H}$, are contained in certain quadrics in $\mathrm{Im}\mathbb{H}$.
\end{lem}

\begin{proof} The  surface $CD C^{-1}$ is a subset of the sphere $\{q\in\mathrm{Im}\mathbb{H}:|q|=|D|\}$. The  surface $ E F^{-1}$ is a rational surface in $\mathbb{R}^3$ of degree at most $1$ in each variable, hence a subset of a quadric. Indeed, denote $E=:iX_1+jX_2+kX_3$, $F=:X_4$. The $10$ polynomials $X_mX_n$, where $1\le m\le n\le 4$, belong to $9$-dimensional space $\mathbb{R}_{22}$. Thus there is a linear dependence between them. It gives the equation of the required quadric containing the surface $EF^{-1}$.
\end{proof}

\begin{lem}\label{l-deg4} 
If the surface $AB^{-1}$, 
where $A,B\in\mathbb{H}_{11}-\{0\}$, is contained in $\mathrm{Im}\mathbb{H}$, then it is contained in an irreducible algebraic surface of degree at most $4$.
\end{lem}

\begin{proof} By Lemma~\ref{l-3folddeg4} the hypersurface $\tilde A\tilde B^{-1}$, and hence $AB^{-1}$, is contained in an irreducible algebraic hypersurface of degree $\le 4$. One of the components of the intersection of the algebraic hypersurface with the hyperplane $\mathrm{Im}\mathbb{H}$ is the required surface unless the algebraic hypersurface coincides with the hyperplane.
In the latter case the assumptions of Lemma~\ref{l-3fold-R3} are satisfied. By Lemmas~\ref{l-3fold-R3} and~\ref{l-quadrics} the surface $AB^{-1}$ is a M\"obius transformation of a subset of quadric, hence has degree $\le 4$.
\end{proof}

Now the following folklore lemma allows us to prove a particular case of Theorem~\ref{prop21}. 

\begin{lem}\label{l-cyclide} 
Assume that an irreducible algebraic surface of degree $\le 4$ in $\mathbb{R}^3$ does not contain the origin. If its inversion in the unit sphere has degree $\le 4$, then both surfaces are Darboux cyclides.
\end{lem}

\begin{proof} Denote $G(x,y,z):=x^2+y^2+z^2$. Let $A(x,y,z)=0$ and $B(x,y,z)=0$ be the minimal degree equations of the initial surface and its inversion respectively, $A(0,0,0)\ne 0$. Let us prove that $B(x,y,z)=0$ is a Darboux cyclide; then the initial surface is too. We may assume that $\deg B>2$.

Since the surfaces are symmetric with respect to $S^2$, it follows that $A(\frac{x}{G(x,y,z)},\frac{y}{G(x,y,z)},\frac{z}{G(x,y,z)})=0$ is the same surface as $B(x,y,z)=0$. By the Hilbert Nullstellensatz it follows that
$\left(G^4A(\frac{x}{G},\frac{y}{G},\frac{z}{G})\right)^l$
is divisible by $B$ for some $l$.
Since $B$ is irreducible, $G^4A(\frac{x}{G},\frac{y}{G},\frac{z}{G})=BC$ for some polynomial $C$. 

Expand $A=A_0+A_1+\cdots+A_a$, $B=B_0+B_1+\cdots+B_b$, $C=C_0+C_1\cdots+C_c$, where $2<b\le 4$ and $A_l$, $B_l$, $C_l$ are homogeneous of degree $l$. We get $BC=G^4A_0+G^{3}A_1+\dots +G^{4-a}A_a$.

Comparing the highest-degree terms of the left- and the right-hand side we get $B_bC_c=G^4A_0$ because $A_0=A(0,0,0)\ne 0$. Since $G$ is irreducible, it follows that $B_b$ is a power of $G$ times a constant which we set to $1$. Since $2<b\le 4$ it follows that $b=c=4$,  $B_4=G^2$, $C_4=A_0G^2$.

Comparing the degree $7$ terms of the right- and the left-hand side we get $B_4C_{3}+B_3C_4=G^3A_1$, that is, $G^2C_{3}+G^2A_0B_3=G^3A_1$. Hence $C_{3}+A_{0} B_3$ is divisible by $G$.
Comparing the degree $6$ terms we get $B_4C_{2}+B_3C_3+B_2C_4=G^2A_2$, that is, $ G^2C_{2}+B_3C_3+A_0G^2B_2=G^2A_2$. Hence $B_3$ or $C_3$ is divisible by $G$. Therefore they both are divisible. We have proved that $B_4$ and $B_3$ are divisible by $G^2$ and $G$ respectively, i.e., $B(x,y,z)=0$ is a Darboux cyclide, set (D) in Main Theorem~\ref{mainthm}.
\end{proof}

\begin{exmp} (A.~Gaifullin, private communication) The irreducible degree $4$ surfaces given by the equations $(x^2+y^2+z^2)y^2-5x^3+z^2=0$ and
$(x^2+y^2+z^2)z^2-5x^3+y^2=0$ are symmetric with respect to the unit sphere but are not Darboux cyclides. Here both surfaces do contain the origin.
\end{exmp}


\begin{lem}\label{prop11}
If the surface  $\Phi(u,v)=A(u,v)B(u,v)^{-1}$, where $A,B\in\mathbb{H}_{11}-\{0\}$, is contained in $\mathbb{R}^3$ (respectively, in $S^3$) then it is a subset of a Darboux cyclide (respectively, an intersection of $S^3$ with another $3$-dimensional quadric in $\mathbb{R}^4$).
\end{lem}

\begin{proof}
In $ \R^3$, translate the surface so that it does not contain the origin. The resulting surface still has form $AB^{-1}$, and the lemma follows from Lemmas~\ref{l-deg4} and~\ref{l-cyclide} applied to the pair $AB^{-1}$ and $-BA^{-1}$.
We assume that still $A\ne 0$; otherwise there is nothing to prove.

Now consider a surface in $S^3$. Project it stereographically to $\mathbb{R}^3$.
The resulting surface has the form $(A+B)(A-B)^{-1}$. By the lemma for $\mathbb{R}^3$ it is a Darboux cyclide.  By \cite[Section~2.2]{PSS11} the initial surface is an intersection of $S^3$ with another quadric.
\end{proof}

Let us proceed to the general case of Theorem~\ref{prop21}. 



\begin{proof}[Proof of Theorem~\ref{prop21} for $\Phi\subset\mathbb{R}^3$] 
If $A=\mathrm{const}$ then 
by Lemma~\ref{prop11} $\Phi=(\overline{A}^{-1}B)\overline{C}^{-1}$ is a subset of a Darboux cyclide. 
Analogously, if $C=\mathrm{const}$ then
$\Phi=-\overline{\Phi}=(-C^{-1}\overline{B})A^{-1}$ 
is a Darboux cyclide.

Assume further that $A,C\ne\mathrm{const}$. Since $AC\in\mathbb{H}_{11}$ it follows that either $A\in\mathbb{H}_{01}$ and $C\in\mathbb{H}_{10}$ or vice versa. Assume the former without loss of generality.

By left division of both $\overline{A}$ and $B$ by the leading coefficient of $\overline{A}$ we may achieve $\overline{A}(v)=v+a$ for some $a\in\mathbb{H}$. Analogously, assume that $\overline{C}(u)=u+c$ for some $c\in\mathbb{H}$. Performing the change of variables $v\mapsto v-\mathrm{Re}\,a$ and $u\mapsto u-\mathrm{Re}\,c$  we may achieve $a,c\in\im\H$.
If $a=0$ then $\Phi=B\cdot(uv+cv)^{-1}$ is a subset of a Darboux cyclide by Lemma~\ref{prop11}. The same holds for $c=0$.

Assume further that $a,c\ne 0$. Then we have
\begin{align*}
\Phi(u,v)-\Phi(u,0)-\Phi(0,v)+\Phi(0,0)&=
\overline{A}^{-1}(B-\overline{A}a^{-1}B(u,0)-B(0,v)c^{-1}\overline{C}
+\overline{A}\Phi(0,0)\overline{C})\overline{C}^{-1}\\
&=(v+a)^{-1}buv(u+c)^{-1}
\end{align*}
for some $b\in\mathbb{H}$ because the left-hand sides vanishes identically for $u=0$ or $v=0$.

If $b=0$ then $\Phi(u,v)=\Phi(u,0)+\Phi(0,v)-\Phi(0,0)$.
By Lemma~\ref{l-axial} the curves $\Phi(u,0)$ and $\Phi(0,v)-\Phi(0,0)$ are either circles, or lines, or points. If they both are circles then $\Phi$
is a subset of a Euclidean translational surface~(E); otherwise $\Phi$ is a subset of a quadric.

Assume that $b\ne 0$. By the above $(v+a)^{-1}buv(u+c)^{-1}\in\im\H$ for each $u,v\in\mathbb{R}$. Thus $\mathrm{Re}((v+\bar a)b(u+\bar c))=0$, hence $b\in \im\H$, $b\perp a$, $b\perp c$, $a\times b\perp c$. Since $a,b,c\ne 0$ we get $a\parallel c$.
By Lemma~\ref{l-axial} the curves $u=\mathrm{const}$ and $v=\mathrm{const}$ are circles (or points) whose planes are orthogonal to the vector $a\parallel c$. 
Thus all these circles and hence the surface $\Phi(u,v)$ are contained in one plane.
\end{proof}

\begin{proof}[Proof of Theorem~\ref{prop21} for $\Phi \subset S^3$]
As in the previous proof, we may assume that $A\in\mathbb{H}_{01}$, $C\in\mathbb{H}_{10}$ and $A,C\ne\mathrm{const}$. Since $\Phi \subset S^3$ it follows that
$|A^{-1}BC^{-1}| = 1$ and $|B(u,v)|^2 = |A(v)|^2|C(u)|^2$.
By Splitting Lemma~\ref{l-splitting-basic} there exist $D(v)\in\mathbb{H}_{01}$ and $E(u)\in\mathbb{H}_{10}$
such that  $B(u,v)$ splits:
$B(u,v) = D(v) E(u)$ or $B(u,v) = E(u) D(v)$. Since $|A(v)|\cdot |C(u)|=|B(u,v)|=|D(v)|\cdot |E(u)|$,  without loss of generality we may assume that
 $|D| = |A|$, $|E| = |C|$.

In the case when $B=DE$ we have $\Phi = \overline{A}^{-1}DE\overline{C}^{-1}= A\,|A|^{-2}\,|D|^2\,\bar D^{-1}E\overline{C}^{-1}=(A\bar D^{-1})(E\overline{C}^{-1})$
is a product of two circles or points $A(v)\bar D(v)^{-1}$ and $E(u)\overline{C}(u)^{-1}$. They are contained in $S^3$ because $|D| = |A|$, $|E| = |C|$. Thus $\Phi$ is a subset of a Clifford translational surface in $S^3$.

In the case when $B=ED$ we have
$
\Phi = \overline{A}^{-1}ED\overline{C}^{-1} =  AE\overline{D}^{-1}\overline{C}^{-1} = ( AE)(\overline{C}
\,\overline{D})^{-1}
$
because $|D|=|A|$. Since $AE, \overline{C}\,\overline{D} \in \H_{11}$ it follows by Lemma~\ref{prop11} that $\Phi$ is contained in the intersection of $S^3$ with another $3$-dimensional quadric.
\end{proof}

\begin{lem} \label{l-clifford} Each Clifford translational surface in $S^3$
can be rotated  (i.e., transformed by an orientation-preserving isometry of $S^3$) so that the stereographic projection of the resulting surface is a Clifford translational surface in $\mathbb{R}^3$.
\end{lem}

\begin{proof}
For any $\alpha,\beta\subset\mathbb{H}$ denote $\alpha\cdot\beta:=\{\,p\cdot q:p\in\alpha,q\in\beta\,\}$. Consider a Clifford translational surface $\alpha\cdot\beta$  in $S^3$, where $\alpha$ and $\beta$ are circles in $S^3$. Take $a,b\in S^3$ such that $\bar a$ and $\bar b$ are orthogonal to the 3- or 2-dimensional real linear spans of $\alpha$ and $\beta$ respectively. Then $\alpha':=\{a\}\cdot\alpha$ and $\beta':=\beta \cdot \{b\}$ are circles in $S^2$. The rotation $q\mapsto aqb$ takes the surface $\alpha\cdot\beta$ to the surface $\alpha'\cdot\beta'$. The stereographic projection of the latter is computed directly: for each $p\in\alpha'$, $q\in\beta'$, $p+q\ne0$, we have
$$
(1+pq)(1-pq)^{-1}
=(1+pq)\overline{(1-pq)}\left((1-pq)\overline{(1-pq)}\right)^{-1}
=(pq-qp)|p+q|^{-2}=2(p\times q)|p+q|^{-2}
$$
because $p^2=q^2=-1$, once $p,q\in S^2$. We get a Clifford translational surface in $\mathbb{R}^3$.
\end{proof}

\begin{proof}[Proof of Theorem~\ref{cor-3D}]
Assume that a surface $\Psi\subset\mathbb{R}^3$ is a M\"obius transformation of surface~\eqref{eq-4DBconj}.
Since a M\"obius transformation takes a hyperplane to either a hyperplane or a $3$-dimensional sphere it follows that the latter surface 
is contained either in a hyperplane or in a $3$-dimensional sphere.

Perform a M\"obius transformation $q\mapsto aqc+b$  (
a similarity), where $a,b,c\in\H$, $a,c\ne 0$, taking the obtained hyperplane (respectively, the $3$-dimensional sphere) to $\im\H$ (respectively, to $S^3$). It takes the surface $\bar A^{-1}B\bar C^{-1}$ to the surface  $(\bar Aa^{-1})^{-1}(B+\bar Aa^{-1}bc^{-1}\bar C) (c^{-1}\bar C)^{-1}$
again of form~\eqref{eq-4DBconj}.

By Theorem~\ref{prop21} the resulting surface 
is (D), or (E), or intersection of $S^3$ with a quadric, or Clifford translational in $S^3$.
In the case next to the latter perform the stereographic projection to get (D). In the latter case perform the rotation given by Lemma~\ref{l-clifford} and then the stereographic projection to get (C).
In all the cases the resulting surface is a M\"obius transformation of $\Psi\subset\mathbb{R}^3$. 
\end{proof}



\vspace{-0.6cm}

\section{Parametrization results}\label{sec:proofs2}
\label{sec-parametrization}

In this section we prove the results on parametrization of surfaces containing two conics or circles through each point, in particular, Theorem~\ref{haupt}. The proof uses well-known methods from algebraic geometry and goes along the lines of \cite{schicho:2001}, where the case of conics is considered.
In comparison to \cite{schicho:2001} we have added some details 
to make it accessible to nonspecialists and eliminated some unnecessary steps.
All the lemmas below are known, but we could not find appropriate references. Some lemmas use notation defined in previous statements (but not the proofs). E.g., in Lemmas~\ref{l-cover}--\ref{l-final}
and \ref{l-reduction}--\ref{l-make-real}, $\Phi$ is a surface satisfying the assumptions of Theorems~\ref{schicho} and~\ref{haupt} respectively.

We use the following notions.
Let $P^n$ be the $n$-dimensional complex projective space with the homogeneous coordinates $x_1:\dots:x_{n+1}$. Throughout we use the standard topology in $P^n$ (not the Zariski one).
An \emph{analytic surface} in $P^n$ is the image of an injective complex analytic map from a domain in $\mathbb{C}^2$  into $P^n$ with nondegenerate differential at each point.
An \emph{algebraic subset} $X\subset P^n$ is the solution set of some system of algebraic equations. Algebraic subsets of dimension $1$ and $2$ are called \emph{projective algebraic surfaces} and \emph{algebraic curves}, respectively. A \emph{conic} is an irreducible degree $2$ algebraic curve in $P^n$. Recall that the set of all conics and pairs of lines (possibly coincident) in $P^n$ is naturally identified with an algebraic subset of $P^N$ for some large $N$ (depending on $n$). The subset is called \emph{the variety of all conics in $P^n$}. A conic \emph{analytically depending} on a point is a complex analytic map of an analytic surface in $P^n$ into the variety of all conics in $P^n$. An \emph{analytic family} of conics is a complex analytic map $v\mapsto\alpha_v$ of a domain in $\mathbb{C}$ into the variety of all conics in $P^n$ such that none of $\alpha_v$ is a pair of lines. 
The image of this map is also called a \emph{family} of conics.

Theorem~\ref{haupt} is deduced from the following one, which is essentially \cite[Theorem~11]{schicho:2001}.


\begin{thm}\label{schicho} 
Assume that through each point of
an analytic surface $\Phi$ in 
a domain in $P^n$
one can draw two transversal conic sections intersecting each other only at this point (and analytically depending on the point) such that their intersections with the domain are contained in the surface.
Assume that through each point in some dense subset of the surface one can draw only finitely many conic sections such that their intersections with the domain are contained in the surface. Then
the surface is algebraic and has a parametrization (possibly besides a one-dimensional subset) \vspace{-0.1cm}
\begin{equation}\label{eq-conical}
\Phi(u,v)=X_{1}(u,v):\dots:X_{n+1}(u,v)
\end{equation}
\vspace{-0.2cm}
for some $X_{1},\dots,X_{n+1}\in\mathbb{C}_{22}$
such that
the conic sections are the curves $u=\mathrm{const}$ and $v=\mathrm{const}$.
%
\end{thm}

Conversely, one can see immediately that almost every surface~\eqref{eq-conical} contains two conic sections $u=\mathrm{const}$ and $v=\mathrm{const}$ through almost every point. Theorem~\ref{schicho} in particular implies that each surface containing two conic sections through each point is contained in an projective subspace of dimension at most $8$ and has the degree at most $8$ (by standard elimination of variables).

Let us prove Theorem~\ref{schicho}.
In what follows $\Phi$ is a surface 
satisfying the assumptions of the theorem unless otherwise indicated. Denote by $\alpha_P$ and $\beta_P$ the two conics drawn through a point~$P\in\Phi$.

\begin{lem} \label{l-cover}
\footnote{This lemma is stated and proved in a slightly stronger form than in the published version of the paper.}
Under the assumptions of Theorem~\ref{schicho}, possibly except that $\alpha_P\cap\beta_P=\{P\}$, there are
two analytic families of conics $\alpha_v$, $\beta_u$, 
and a domain $\Omega\subset P^n$ such that $\bigcup_{v}\alpha_v\cap\Omega=
\bigcup_{u}\beta_u\cap\Omega=\Phi\cap\Omega\ne\emptyset$, and inside $\Omega$, each pair $\alpha_v$, $\beta_u$ intersects transversely at a unique point~$P(u,v)=\alpha_v\cap\beta_u\cap\Omega$. Moreover, 
$\alpha_{P(u,v)}=\alpha_v$, $\beta_{P(u,v)}=\beta_u$.
\end{lem}

\begin{proof}
Take a point $P_0\in\Phi$. Draw the two conics $\alpha_0:=\alpha_{P_0}$ and $\beta_0:=\beta_{P_0}$ in the surface through the point. Through each point $P\in\alpha_0\cap\Phi$ draw another conic $\beta_{P}$ in the surface. We get an analytic family  of conics $\beta_u$. Analogously we get an analytic family  of conics $\alpha_v$.

By the assumptions of Theorem~\ref{schicho} the conics $\alpha_0$ and $\beta_0$  intersect transversely at a unique point. By the continuity there is $\epsilon>0$ such that for $|u|,|v|<\epsilon$ the conics $\alpha_v$ and $\beta_u$ intersect transversely at a unique point $P(u,v)$, and $P(u,v)\in\Phi$. Take a sufficiently small neighborhood $\Omega$ of the point $\alpha_0\cap\beta_0$ in $P^n$. Then $\bigcup_{v}\alpha_v\cap\Omega=
\bigcup_{u}\beta_u\cap\Omega=\Phi\cap\Omega\ne\emptyset$.

It remains to show that 
$\alpha_{P(u,v)}=\alpha_v$, $\beta_{P(u,v)}=\beta_u$.
Assume that, say, $\alpha_{P(u,0)}\ne\alpha_{0}$ for some $u$ (the conics $\alpha_v$ are equitable in the following argument). By the continuity there is $\epsilon>0$ such that $\alpha_{P(u,v)}\ne\alpha_{v}$ for each $|v|<\epsilon$ and this particular $u$. Let us show that for each such $v$ there is $\delta(v)>0$ such that $\alpha_{P(u,v)}\ne\alpha_{v}$ for each $0<|u|<\delta(v)$. Indeed, otherwise for some $|v|<\epsilon$ and each $\delta>0$ there exist $u$  such that $0<|u|<\delta$ and $\alpha_{P(u,v)}=\alpha_{v}$. By the uniqueness theorem for analytic functions it follows that $\alpha_{P(u,v)}=\alpha_{v}$
for each $u$ and this particular $v$ with $|v|<\epsilon$. This contradicts to the choice of $\epsilon$. Thus the required $\delta(v)$ exists. Consider the following $2$ cases.

Case 1: for all sufficiently small $u,v$ the conic $\alpha_{P(u,v)}$ belongs to the family $\alpha_v$.
Consider the envelope of the family $\alpha_v$.
For that purpose, introduce analytic coordinates $(x,y)$ on $\Phi\cap \Omega$ so that $\beta_0$ has the equation $x=0$, while $\alpha_v$ passes through $(0,v)$ and has an analytic equation $a_v(x,y)=0$. By definition, an \emph{envelope} is the set of points $(x,y)$ such that for some $|v|<\epsilon$ we have
$$
  a_v(x,y)=0 \qquad\text{and}\qquad
  \frac{\partial a_v}{\partial v}(x,y)=0.
$$
(An envelope may slightly depend on the choice of the equation $a_v(x,y)=0$ of $\alpha_v$.) Since the intersection point  $P(u,v)\in\alpha_{P(u,v)}\cap \alpha_v$ tends to $P(0,v)$ as $u\to 0$ for fixed $v$ and $\alpha_{P(u,v)}\ne\alpha_{v}$ for $0<|u|<\delta(v)$, it follows that the envelope passes through $P(0,v)$ \cite[\S5.9]{bruce-giblin:1984}. Since this holds for all sufficiently small $v$, the envelope contains an open subset of the conic $\beta_0$. But the conic $\beta_0$ is transversal to the family $\alpha_v$ rather than tangent to it, a contradiction:
$$
0=\frac{da_v}{dv}(0,v)=\frac{\partial a_v}{\partial v}(0,v)+
\frac{\partial a_v}{\partial y}(0,v)=0+\frac{\partial a_v}{\partial y}(0,v)\ne 0.
$$

Case 2: there are arbitrarily small $u,v$ such that the conic $\alpha_{P(u,v)}$ does not belong to the family $\alpha_v$. Then the image of the analytic map $P\mapsto \alpha_P$ in the variety of all conics in $P^n$ cannot be $1$-dimensional. Hence it is $2$-dimensional. Then there are infinitely many conics $\alpha$ through each point in an open subset of $\Phi$ such that $\alpha\cap\Omega\subset\Phi$. This contradicts to one of the assumptions of Theorem~\ref{schicho}.

Thus the families $\alpha_v$ and $\beta_u$
are the required.
\end{proof}

Let us summarize informally what we have achieved.
Clearly, we may assume that $v\mapsto\alpha_v\cap\beta_0$ and $u\mapsto\alpha_0\cap\beta_u$ are given by quadratic polynomials (although formally we do not assume that further). The required parametrization is then obvious: $\Phi(u,v)=\alpha_v\cap\beta_u$; the difficulty is to prove that it indeed has form~\eqref{eq-conical}. A hidden obstacle is intersection of conics in one family, like in the following example. To rule such possibility out, we need heavy machinery of algebraic geometry.

\begin{exmp} Let $\Phi(x,y)=x:y:xy:(1+x^2+y^2)$, let
$\beta_u$ and $\alpha_v$ be the images of the lines $x=u$ and $y=f(v)x+v$, where $f(v)$ is an irrational analytic function.
In an appropriate domain, the surface $\Phi$ and the conics $\beta_u$, $\alpha_v$ satisfy all the assumptions of Theorem~\ref{schicho} except the finiteness of the number of conics through each point. But the map $(u,v)\mapsto \alpha_v\cap \beta_u$ does not have form~\eqref{eq-conical}. 
\end{exmp}


\begin{lem} \label{l-analytic} If an open subset of an analytic surface in $P^n$ or $\mathbb{R}^3$ is contained in an algebraic surface, then the whole analytic surface is contained in the algebraic one.
\end{lem}

\begin{proof}
Let $X_1(u,v):\cdots:X_{n+1}(u,v)$, where $X_1,\dots,X_n$ are analytic, be a parametrization of the analytic surface and $A_k(x_1,\dots,x_{n+1})=0$, where $k=1,\dots,m$, be a system of equations defining the algebraic surface. Then the analytic function $A_k(X_1(u,v),\dots,X_{n+1}(u,v))$ vanishes on an open set, hence everywhere, by the uniqueness theorem for analytic functions.
\end{proof}

\begin{lem} \label{l-algebraic} The surface $\Phi$ is contained in an irreducible algebraic surface $\hat\Phi$. The family of conics $\alpha_v$ is contained in an irreducible algebraic curve in the variety of all conics.
\end{lem}

\begin{proof} 
The conics $\beta_u$ given by Lemma~\ref{l-cover} have at most $4$ common points because they do not all coincide. Thus we may assume that they do not have common points inside the domain $\Omega$ (by Lemma~\ref{l-analytic} one can consider the part of the surface in a smaller domain, if necessary). 
Then by the analyticity each point of $\Omega$ belongs to at most countable number of conics $\beta_u$.
Take a sufficiently small $\epsilon>0$ such that $\alpha_v\cap\beta_u\subset\Omega$ for each $|u|,|v|\le\epsilon$.

Consider the set $\gamma$ of all conics in $P^n$ intersecting each conic $\beta_u$, where $|u|\le\epsilon$, transversely at a unique point, which belongs to the domain $\Omega$.
Let us show that the set $\gamma$ is an open subset of an algebraic subvariety of the variety of all conics in $P^n$ (with respect to ordinary topology). Indeed, the set of all conics intersecting a fixed conic $\beta_u$ is clearly an algebraic subvariety. The set of all conics intersecting each conic $\beta_u$, where $|u|\le\epsilon$, is the intersection of infinitely many such algebraic subvarieties, and hence also an algebraic subvariety. To prove that $\gamma$ is open in the latter subvariety, take arbitrary $\hat\alpha\in\gamma$ and $|\hat u|\le\epsilon$. Then there are open subsets $\gamma_{\hat u}\ni \hat\alpha$ of the subvariety and $\delta_{\hat u}\ni \hat u$ of $\mathbb{C}$ such that each conic $\alpha\in\gamma_{\hat u}$ intersects each $\beta_u$, where $u\in\delta_{\hat u}$, transversely at a unique point, which belongs to~$\Omega$. By compactness the disc $|u|\le \epsilon$ is the union of finitely many sets of the form $\delta_{\hat u}$. Then an intersection of finitely many sets of the form $\gamma_{\hat u}$ is a neighbourhood of $\hat\alpha$ contained in $\gamma$.

Let $\hat\gamma$ be the irreducible component of the minimal algebraic subset containing $\gamma$ such that $\hat\gamma$ contains all the conics $\alpha_v$. (Since the family $\alpha_v$ is analytic, it cannot ``switch'' between irreducible components, cf.~Lemma~\ref{l-analytic}.) Let $\hat\Phi$ be the union of all the conics of $\hat\gamma$ (including the ones degenerating to lines and pairs of lines).

Let us show that the algebraic subset $\hat\Phi$ is the required irreducible algebraic surface. To estimate its dimension, take an arbitrary conic $\alpha\in\gamma$. Since each point of $\Omega$ belongs to at most countable number of conics $\beta_u$ it follows that the conic $\alpha$ has uncountably many intersection points with $\bigcup_{u}\beta_u\cap\Omega=\Phi\cap\Omega$.
Hence there is a domain $\Omega'\subset \Omega$ such that $\alpha\cap\Omega'\subset \Phi\cap \Omega'$, because $\Phi$ is analytic. Thus $\bigcup_{\alpha\in\gamma}\alpha$ can only be $2$-dimensional. Since $\gamma$ is open in $\hat\gamma$, it follows that 
$\dim\hat\Phi=\dim\bigcup_{\alpha\in\hat\gamma}\alpha=2$.
By Lemma~\ref{l-analytic} the surface $\hat\Phi$ contains $\Phi$ because $\hat\Phi$ contains the open subset $\bigcup_{v}\alpha_v\cap\Omega=\Phi\cap\Omega$ of~$\Phi$.

Let us show that $\dim\hat\gamma=1$, i.e., $\hat\gamma$ is the required algebraic curve in the variety of all conics. We have $\dim\hat\gamma\ge 1$ because $\gamma$ contains all the conics $\alpha_v$, where $|v|\le\epsilon$. Assume that $\dim\hat\gamma\ge 2$. Then there is a domain $\Omega'\subset \Omega$ such that through each point of $\Phi\cap \Omega'$
there are infinitely many conics $\alpha\in\hat\gamma$ satisfying $\alpha\cap\Omega'\subset \Phi\cap \Omega'$. This contradicts to the assumptions of Theorem~\ref{schicho}. Thus $\dim\hat\gamma=1$.
\end{proof}

\begin{rem} \label{rem-infinity} 
If we drop the assumption that the number of conics through certain points is finite in Theorem~\ref{schicho} then $\Phi$ is contained in an algebraic surface $\hat \Phi$ and $\alpha_v$ is contained in an algebraic subvariety $\hat\gamma$ of the variety of all conics in $P^n$ such that $\bigcup_{\alpha\in\hat\gamma}\alpha=\hat\Phi$. The proof is the same except that one should drop the 'moreover' part of Lemma~\ref{l-cover}, the last $4$ paragraphs of the proof of Lemma~\ref{l-cover}, and the last paragraph of the proof of Lemma~\ref{l-algebraic}.
\end{rem}




The following definitions are equivalent to the commonly used one for the particular situation considered in the paper.

Let $X,Y\subset P^n$ be algebraic subsets. A \emph{rational map} $f\colon X\dasharrow Y$ is a map defined on a subset of $X$ being the complement of an algebraic subset of smaller dimension, assuming values in the set $Y$, and given by polynomials in homogeneous coordinates of $P^n$. (Dashes in the notation remind that a rational map may not be defined everywhere in $X$.)
A continuous extension of $f$ to a larger subset of $X$ is still called a \emph{rational map} (it may not be given by the same polynomials).
If the restriction of $f$ to complements of certain algebraic subsets in $X$ and $Y$ of smaller dimension is bijective, and the inverse map is rational as well, then $f$ is called a \emph{birational map}. A rational map $X\dasharrow P^1$ is called a \emph{rational function}. A surface $\Psi$ is \emph{rational}, if there is a birational map~$P^1\times P^1\dasharrow\Psi$.

We need the following well-known result for which we could not find any direct reference. (We have found several more general results in the literature but each time the proof that the inverse map is rational, the only assertion we need, was omitted.)
\footnote{E.g., on issues with this very assertion in books by Hartshorne and Goerz--Wedhorn see:\\  http://math.stackexchange.com/questions/70293/why-does-the-definition-of-an-open-subscheme-open-immersion-of-schemes-allow-f,
http://mathoverflow.net/questions/182044/why-is-the-inverse-of-a-bijective-
rational-map-rational
}

\begin{lemma}\label{l-postulate} If the restriction of a rational map $f$ from a projective algebraic surface to $P^1\times P^1$ to the complements of some $1$-dimensional subsets is bijective, then the inverse map is rational as well, and hence birational.
\end{lemma}

\begin{proof} (S. Orevkov, private communication) 
Let the surface be a subset of $P^n$.
For a generic $u\in P^1$ the preimage $f^{-1}(u\times P^1)$ is the algebraic curve $\beta$ in $P^n$ (possibly with finitely many points missing) defined by the equation $\mathrm{pr}_1f(x_1,\dots,x_{n+1})=u$, where $\mathrm{pr}_1\colon P^1\times P^1\to P^1$ is the first projection. 

Restrict the map $f$ to the preimage. We get a bijective rational map between the curves $\beta$ and $u\times P^1$ with finitely many points missing. Clearly, it extends to a rational homeomorphism $\beta\to u\times P^1$. Then by the classification of algebraic curves there is a birational map $\beta\to P^1$. Identify $\beta$ and $u\times P^1$ with $P^1$. Consider the graph of the rational map $\beta\to u\times P^1$ as a subset of $P^1\times P^1$. By elimination of variables it follows that the graph is the zero set of some irreducible polynomial $P\in \mathbb{C}[u,v]$. Since the map $\beta\to u\times P^1$ is bijective it follows that the polynomial $P$ has degree $1$ in each variable and hence the inverse map $u\times P^1\to\beta$ is also rational. 

This implies that the inverse map $f^{-1}(u,v)$ is rational in the variable $v$ for fixed generic $u$. Analogously, $f^{-1}(u,v)$ is rational in $u$ for fixed generic $v$. 
Thus $f^{-1}(u,v)$ is rational. 
\end{proof}

We prove the next lemma by reduction to its 3-dimensional version from \cite{schicho:2001}. A direct proof exposed for nonspecialists is given in the appendix.

\begin{lemma} \label{l-rational} The surface $\hat\Phi$ is rational.
\end{lemma}

\begin{proof} If $n>3$ then consider a generic projection of $\hat\Phi$ to $P^3$.
The image still contains $2$ conics through almost every point. By \cite[Theorem~3]{schicho:2001} it is rational. The restriction of the projection to the complements to appropriate $1$-dimensional subsets is bijective, hence by Lemmas~\ref{l-algebraic} and~\ref{l-postulate} it is birational. Hence $\hat\Phi$ is rational as well.
\end{proof}

The proof of the following lemma requires much more notions (cf.~\cite[\S III.1]{Shafarevich}) and can be really hard for one without algebro-geometric background.

\begin{lem}\label{l-linear-pencil} Each of the families $\alpha_v$ and $\beta_u$ 
consists of level sets of some rational function~$\hat\Phi\dasharrow P^1$.
\end{lem}

\begin{proof} 
By the Hironaka theorem 
(or by an earlier Zariski theorem sufficient in our situation) the surface $\hat\Phi$ has a \emph{desingularization} $d:\tilde\Phi\to\hat\Phi$, i.e., a proper birational map from a smooth projective algebraic surface $\tilde\Phi$ to the surface $\hat\Phi$.
An \emph{effective divisor} on $\tilde\Phi$ is a formal linear combination of irreducible algebraic curves on $\tilde\Phi$ with positive integer coefficients.
Closure of the preimage of an algebraic hypersurface under a rational map from $\tilde\Phi$ to a projective space can be considered as an effective divisor, once the preimage is $1$-dimensional and one counts the irreducible components with their multiplicities.
The collection of preimages of all the hyperplanes under a rational map is called a 
\emph{linear family} of divisors.
In particular, a $1$-dimensional \emph{linear family} is the collection of level sets of a rational function. 
Two divisors are \emph{linear equivalent}, 
if they are contained in a $1$-dimensional linear 
family.
Each linear equivalence class is a linear family  \cite[\S III.1.5]{Shafarevich}.

The \emph{intersection} $D_1\cap D_2$ of two divisors $D_1$ and $D_2$ on $\tilde\Phi$ is the number of their intersection points counted with multiplicities (once the number of intersection points is finite).
Two divisors are \emph{numerically equivalent}, if their intersections with
each algebraic curve on $\tilde\Phi$ are equal, once the number of intersection points is finite. The \emph{degree} of a divisor is the sum of the degrees of the irreducible components counted with multiplicities. The degree of a divisor equals the intersection of the divisor with a generic hyperplane section of $\tilde \Phi$.

Let us sketch the proof that on a rational surface $\tilde\Phi$ numerical equivalence of effective divisors implies linear equivalence (this is yet another well-known fact for which we could not find a direct reference). Indeed, it suffices to prove that the intersection form on the space of linear equivalence classes of effective divisors on $\tilde\Phi$ is nondegenerate. For $P^1\times P^1$ the intersection form has matrix $\left(\begin{smallmatrix}0&1\\1&0\end{smallmatrix}\right)$ in a natural basis \cite[Example~I.9b]{Beauville}, hence is nondegenerate. By elimination of indeterminacy \cite[Theorem~II.7]{Beauville}, $\tilde\Phi$ is obtained from $P^1\times P^1$ by a sequence of \emph{blow-ups} and their inverses (see \cite[\S II.1]{Beauville} for a definition). By \cite[Proposition~II.3]{Beauville} a blow-up just adds an unisotropic orthogonal one-dimensional space to the space of linear equivalence classes of effective divisors and thus preserves the nondegeneracy of the form.

Assume to the contrary that one of the analytic families $\alpha_v$ and $\beta_u$, say, the first one, is not linear.
By the Hironaka theorem the inverse map $d^{-1}\colon \hat\Phi\dasharrow\tilde \Phi$ is defined everywhere except a finite set. Thus the pullback $d^{-1}\alpha_v$ is an analytic family of algebraic curves. We use the notation $d^{-1}\alpha_v$ for the family of (closed) algebraic curves (not to be confused with the set of preimages $d^{-1}(\alpha_v)$ being algebraic curves possibly with a finite number of points missing).
For an algebraic curve $\beta\subset\tilde\Phi$ distinct from
each ${d^{-1}\alpha_v}$
the intersection $\beta\cap {d^{-1}\alpha_v}$ continuously depends on $v$, and hence is constant. Thus each two curves of the family ${d^{-1}\alpha_v}$ are numerically, hence linearly, equivalent. Thus the family ${d^{-1}\alpha_v}$ is contained in a linear family of effective divisors in $\tilde\Phi$. Since $\alpha_v$ is nonlinear it follows that ${d^{-1}\alpha_v}$ is nonlinear, hence the ambient linear family has dimension~$\ge 2$.

The image of the latter linear family under the projection $d\colon\tilde\Phi\to\hat\Phi$ is at least $2$-dimensional (not necessarily linear) family $\gamma$ of effective divisors on $\hat\Phi$.
All the divisors of the family $\gamma$ have the same degree because their pullbacks are numerically equivalent. (Indeed, for two divisors $D_1,D_2\in\gamma$, and a general position hyperplane $H\subset P^n$ we have $\deg D_1=H\cap D_1=d^{-1}H\cap d^{-1}D_1=d^{-1}H\cap d^{-1}D_2=H\cap D_2=\deg D_2$.) 
Since the curves $\alpha_v$ are conics it follows that the degrees of all these divisors are $2$.
Therefore the divisors of the 
family $\gamma$ are either conics or pairs of lines or lines of multiplicity $2$.  
This is possible only if there are infinitely many conic sections or lines through each point in an open subset of the surface, which contradicts to the assumptions of Theorem~\ref{schicho}. Thus both families $\alpha_v$ and $\beta_u$ must be linear.
\end{proof}

Now we use the assumption of Theorem~\ref{schicho} that $\alpha_P$ and $\beta_P$ have a unique intersection point.

\begin{lemma}\label{l-parametrization}
There is a birational map $\hat\Phi\dasharrow P^1\times P^1$
taking the families $\alpha_v$ and $\beta_u$ to the sets of curves $P^1\times v$ and $u\times P^1$.
\end{lemma}

\begin{proof}
Consider the pair of rational functions given by Lemma~\ref{l-linear-pencil} whose level sets are the two families  of conics $\alpha_v$ and $\beta_u$. The pair of rational functions defines a rational map $\hat\Phi\dasharrow P^1\times P^1$. Since for sufficiently small $|u|, |v|$ each pair $\alpha_v$ and $\beta_u$  has an intersection point it follows that the image of this rational map contains a neighborhood of $(0,0)\in P^1\times P^1$, and thus contains $P^1\times P^1$ besides a $1$-dimensional subset, by~\cite[Theorem~6 in \S I.5.2]{Shafarevich}. Since the intersection point is unique it follows that the point $(u,v)\in P^1\times P^1$ has exactly one preimage for sufficiently small $|u|, |v|$, and hence for almost every $u,v$. Thus  the restriction of the rational map $\hat\Phi\dasharrow P^1\times P^1$ to appropriate subsets is bijective. By Lemma~\ref{l-postulate} the map is birational.
\end{proof}

\begin{lemma}\label{l-final} Assume that a birational map $P^1\times P^1\dasharrow \hat\Phi$ takes the sets of curves $P^1\times v$ and $u\times P^1$ to conics or lines. Write the birational map as
$(u,v)\mapsto X_1(u,v):\dots :X_{n+1}(u,v)$ for some coprime $X_1,\dots,X_{n+1}\in\mathbb{C}[u,v]$.
Then $X_1,\dots,X_{n+1}\in\mathbb{C}_{22}$.
\end{lemma}

\begin{proof}
For a birational map between surfaces, by elimination of indeterminacy \cite[Theorem~II.7]{Beauville} there is always an algebraic curve $\sigma$ such that  outside the curve $\sigma$ the map is injective and has nondegenerate differential.
Fix a generic value of $u$. Denote $X_k(v):=X_k(u,v)$. The curve $X_1(v):\dots :X_{n+1}(v)$ is a conic or a line.
Cut it by a generic hyperplane $\lambda_1 x_1+\dots \lambda_{n+1}x_{n+1}=0$ in $P^n$. The intersection consists of at most $2$ points. By general position they are not contained in the image of $\sigma$.
Since $X_1(u,v),\dots,X_{n+1}(u,v)$ are coprime it follows that
$X_1(v),\dots,X_{n+1}(v)$ have no common roots.
Since the birational map is injective outside $\sigma$ it follows that the equation $\lambda_1 X_1(v)+\dots +\lambda_{n+1}X_{n+1}(v)=0$ has at most $2$ solutions. These solutions have multiplicity $1$ because the birational map has nondegenerate differential outside $\sigma$.
This implies that the polynomials $X_1(v),\dots,X_{n+1}(v)$ have degree at most $2$. Analogously, $X_1(u,v),\dots,X_{n+1}(u,v)$ have degree at most $2$ in $u$, and the lemma follows.
\end{proof}

\begin{proof}[Proof of Theorem~\ref{schicho}] It follows directly by Lemmas~\ref{l-cover}, \ref{l-algebraic}, \ref{l-parametrization}, 
\ref{l-final}.
\end{proof}

\begin{rem} Theorem~\ref{schicho} remains true (with almost the same proof) without the assumption that the number of conic sections through certain points is finite except that then one cannot conclude that the two drawn conic sections $\alpha_P$ and $\beta_P$ are necessarily the curves $u=\mathrm{const}$ and $v=\mathrm{const}$.
\end{rem}

For the proof of Theorem~\ref{haupt} we need the following lemmas and auxiliary Theorem~\ref{thm:c-make-real} presented in the second appendix. In the rest of this section $\Phi\subset S^{n}\subset\mathbb{R}^{n+1}$ is a surface satisfying the assumptions of Theorem~\ref{haupt}.

\begin{lem} \label{l-reduction} The surface $\Phi$ (possibly besides a one-dimensional subset) has parametrization~\eqref{eq-circular},
where $X_{1},\dots,X_{n+2}\in\mathbb{C}_{22}$
satisfy Eq.~\eqref{eq-pythagorean-n}, the circular arcs have form $u=\const$ and $v=\const$,
and $(u,v)$ runs through some (not open) subset of~$\mathbb{C}^2$.
\end{lem}

\begin{proof} It suffices to prove the lemma for an arbitrarily small open subset of $\Phi$; then by the analyticity the whole surface $\Phi$, possibly besides a one-dimensional subset, is going to be parametrized by the same polynomials by Lemma~\ref{l-analytic}. In what follows, each open subset of $\Phi$ we restrict to, is still denoted by $\Phi$.

First extend an appropriate part of $\Phi$ analytically to a complex analytic surface $\hat\Phi$ in a ball $\Omega\subset{P}^n$ such that $\partial\hat\Phi\subset\partial\Omega$. For this purpose, take a real analytic map $f\colon\Psi\to\Phi$ of a planar domain $\Psi\subset\R^2$, having nondegenerate differential $df$ at each point. Fix a point $O\in\Phi$. The Taylor series at $f^{-1}(O)$ provide an extension $\hat f$ of $f$ to the convergence polydisk. Since $df$ is injective at $ f^{-1}(O)$, the restriction $\hat f\colon\hat\Psi\to{P}^n$ to a smaller concentric open polydisc $\hat\Psi$ is injective and has nondegenerate differential at each point. Thus there is an open ball $\Omega$ centered at $O$ such that $\hat f(\partial\hat\Psi)\cap\Omega=\emptyset$. Take $\hat\Phi$ to be the path-connected component of the intersection $ \hat f(\hat \Psi)\cap \Omega$ containing $O$ (this construction is further referred to as the \emph{restriction of the surface to the ball $\Omega$}).
We have $\partial \hat\Phi\subset\partial\Omega$ because $\hat f(\partial\hat\Psi)\cap\Omega=\emptyset$.

Now we check that an appropriate part of $\hat\Phi$ satisfies all the assumptions of Theorem~\ref{schicho}.

The most nontrivial one is that there are at most finitely many \emph{complex} conics through a generic point of $\hat\Phi$. Since $S^{n}$ contains no real lines, 
it follows that $\Phi$ is not covered by a $1$-dimensional family of real lines. 
Then Theorem~\ref{thm:c-make-real} from the appendix implies that $\hat\Phi$ contains finitely many complex conics through a generic point, otherwise $\Phi$ would contain infinitely many circular arcs through a generic point.

Let us fulfill the remaining assumptions. Let $\alpha_P$ and $\beta_P$ be the circles passing through a point $P\in\Phi$.
Extend the real analytic families $\alpha_P$ and $\beta_P$
to complex analytic families $\hat\alpha_P$ and $\hat\beta_P$ of complex conics in ${P}^{n+1}$ parametrized by a point $P$ in a neighborhood of $O$ in $\hat\Phi$. Since $\hat\alpha_O\cap\hat\beta_O=O$, by the continuity it follows that $\hat\alpha_P\cap\hat\beta_P=P$ for each $P$ in a (possibly smaller) neighborhood of $O$ in $\hat\Phi$. Now if $\hat\alpha_O\cap\Omega$ is disconnected, then replace $\Omega$ by a smaller ball centered at $O$ so that it becomes connected.
Restrict $\hat\Phi$ to the new ball $\Omega$. Since an arc of the conic $\alpha_O$ is contained in $\Phi$ and $\partial\hat\Phi\subset\partial\Omega$, by the analyticity it follows that the connected set $\hat\alpha_O\cap\Omega$ is contained in $\hat\Phi$. By the analyticity, the same is true for all $P\in\hat\Phi$ in a neighborhood of $O$. Restricting $\hat\Phi$ to an appropriate smaller ball $\Omega$ once again, we fulfill all the assumptions of Theorem~\ref{schicho}.

So by Theorem~\ref{schicho} the surface $\hat\Phi$ (possibly besides a one-dimensional subset) has parametrization~\eqref{eq-circular} with $X_1,\dots,X_{n+2} \in\mathbb{C}_{22}$ such that the circular arcs have the form $u=\mathrm{const}$ and $v=\mathrm{const}$. (However, the converse is not true: most of the curves $u=\mathrm{const}$ and $v=\mathrm{const}$ are not circular arcs but parts of complex conics in $\hat \Phi$.) Since the surface $\hat \Phi$ is contained in the complex quadric 
extending the sphere $S^n$, it follows that the polynomials satisfy equation~\eqref{eq-pythagorean-n}.
\end{proof}

Let us reparametrize the surface to make the domain of the map $\Phi$ and the coefficients of the polynomials $X_1,\dots,X_{n+2}$ real. This step is much easier than in the proof of Theorem~\ref{thm:c-make-real} because we know that the conics on $\Phi$ have the form $u=\const$ and $v=\const$.

\begin{lem} \label{l-make-real} The surface $\Phi$ (possibly besides a one-dimensional subset) has parametrization~\eqref{eq-circular},
where $X_{1},\dots,X_{n+2}\in\mathbb{C}_{22}$
satisfy Eq.~\eqref{eq-pythagorean-n},
and $(u,v)$ runs through certain open subset of~$\mathbb{R}^2$.
\end{lem}

\begin{proof} Start with the parametrization given by Lemma~\ref{l-reduction}.
Take a point on the surface $\Phi$, which is a regular value of the parametrization map (it exists e.g. by the Sard theorem). Draw two circular arcs of the form $u=\const$ and $v=\const$ through the point. Through another pair of points of the first arc, draw two more circular arcs of the form $v=\const$. Perform a complex fractional-linear transformation of the parameter $v$ so that the second, the third, and the fourth arcs obtain the form $v=0,\pm1$, respectively (we restrict to the part of the surface where the denominator of the transformation does not vanish). Perform an analogous transformation of the parameter $u$ so that $u=0,\pm1$ become circular arcs intersecting the arc $v=0$. After performing the transformations and clearing denominators we get parametrization~\eqref{eq-circular}, where still $X_1,\dots,X_{n+2}\in\mathbb{C}_{22}$ and $(u,v)$ runs through a subset $\Psi\subset\mathbb{C}^2$.

But now all (but one) real points of the arc $u=0$ have real $v$-parameters, and all (but one) real points of the arc $v=0$ have real $u$-parameters. Indeed, in a quadratically parametrized conic, the cross-ratio of any four points equals the cross-ratio of their parameters. Since three (real) points $v=0,\pm 1$ of the arc $u=0$ have real $v$-parameters, it follows that all (but one) points of the arc have real $v$-parameters. The same for the arc $v=0$.

Let us prove that actually $\Psi\subset\R^2$.
Since $\Phi$ is a real analytic surface and $\Phi(0,0)$ is a regular value of the parametrization map, by the inverse function theorem it follows that a neighborhood of $(0,0)$ in $\Psi$ is a real analytic surface as well, injectively mapped to $\Phi$. In what follows the neighborhood is still denoted by $\Psi$. Take $(\hat u,\hat v)\in\Psi$ sufficiently close to $(0,0)$; it may not be real a priori. Then $\Phi(\hat u,\hat v)$ is a point on the surface $\Phi$ sufficiently close to $\Phi(0,0)$. Draw the two circular arcs $u=\hat u$ and $v=\hat v$ through the point $\Phi(\hat u,\hat v)$. Notice that the $v$-parameter along the arc $u=\hat u$ runs through a curve in $\mathbb{C}$, not necessarily an interval in $\R$ a priori. By the continuity it follows that the arc $v=\hat v$ intersects the arc $u=0$ at a real point in the image of $\Psi$. By the injectivity of the map $\Psi\to\Phi$, the intersection point can only be $\Phi(0,\hat v)$. In particular, we get $(0, \hat v)\in\Psi$. Since all (but one) points on the arc $u=0$ have real $v$-parameters, it follows that $\hat v\in\R$. Analogously, $\hat u\in\R$. Thus all $(\hat u,\hat v)\in\Psi$ sufficiently close to the origin are real. Since $\Psi$ is real analytic, by Lemma~\ref{l-analytic} $\Psi\subset\R^2$.
\end{proof}

\begin{lem} \label{l-real} Let $X_1,\dots,X_n\in \mathbb{C}[u,v]$. Assume that for all the points $(u,v)$ from some open subset of $\mathbb{R}^2$ the point $X_1(u,v):\dots:X_n(u,v)$ is real. Then  $X_1=X_1'Y,
\dots,X_n=X_n'Y$ for some real $X_1',\dots,X_n'\in \mathbb{R}[u,v]$ and complex $Y\in\mathbb{C}[u,v]$.			
\end{lem}

\begin{proof}
We may assume that $X_1,\dots,X_n$ do not vanish identically (indeed, if some $X_l=0$ then set $X'_l=0$ and drop $X_l$).
Take generic real $u,v$ from the open set in question. By the assumption of the lemma $X_1(u,v):\dots:X_n(u,v)$ is real. Hence $X_1(u,v)/X_l(u,v)$ is real for each $2\le l\le n$, therefore $X_1(u,v)/X_l(u,v)=\bar X_1(u,v)/\bar X_l(u,v)$. Since this holds for generic real $u,v$ it follows that $X_1/X_l=\bar X_1/\bar X_l$ as rational functions.
Take decompositions
\begin{align*}
X_1&=\lambda_1 Y_{1}^{p_{11}}\bar Y_{1}^{q_{11}}Z_{1}^{r_{11}}\dots Y_{m}^{p_{1m}}\bar Y_{m}^{q_{1m}}Z_{m}^{r_{1m}},\\
& \dots\\
X_n&=\lambda_n Y_{1}^{p_{n1}}\bar Y_{1}^{q_{n1}}Z_{1}^{r_{n1}}\dots Y_{m}^{p_{nm}}\bar Y_{m}^{q_{nm}}Z_{m}^{r_{nm}},
\end{align*}
  into coprime reduced irreducible factors $Y_{1},\dots,Y_{m}\in\mathbb{C}[u,v]-\mathbb{R}[u,v]$, $Z_{1},\dots,Z_{m}\in\mathbb{R}[u,v]$ with the powers $p_{ij},q_{ij},r_{ij}\ge 0$,  and constant factors $\lambda_1,\dots,\lambda_n\in\mathbb{C}-\{0\}$.

Since $\mathbb{C}[u,v]$ is a unique factorization domain, the relation $X_1/X_l=\bar X_1/\bar X_l$ implies that $\lambda_l\bar\lambda_1\in\mathbb{R}$ and $p_{1k}-q_{1k}=\dots=p_{nk}-q_{nk}$ for each $k=1,\dots,m$. Without loss of generality assume that $p_{1k}-q_{1k}\ge 0$ for each $k=1,\dots,m$ (otherwise replace $Y_k$ by $\bar Y_k$ and vice versa). It remains to set
\begin{align*}
 X'_l&:=\lambda_l\bar\lambda_1\left(Y_{1}\bar Y_{1}\right)^{q_{l1}}Z_{1}^{r_{l1}}\dots \relax \left(Y_{m}\bar Y_{m}\right)^{q_{lm}}Z_{m}^{r_{lm}};\\
 Y&:=\bar\lambda_1^{-1}Y_{1}^{p_{11}-q_{11}}\dots Y_{m}^{p_{1m}-q_{1m}}.
\\[-1.4cm]
\end{align*}
\end{proof}

\begin{proof}[Proof of Theorem~\ref{haupt}]
This follows from Lemmas~\ref{l-reduction}--\ref{l-real}.
\end{proof}

The following folklore result is the last one required for the proof of main theorem.

\begin{lem} \label{l-infinite} Assume that through each point of an analytic surface in $\R^3$ or $S^3$ one can draw infinitely many (not nested) circular arcs fully contained in the surface.
Assume that among the infinite number, one can choose two arcs analytically depending on the point. Then the surface is a subset of a sphere or a plane.
\end{lem}

\begin{proof} For a surface in $\mathbb{R}^3$, first perform the inverse stereographic projection of the surface to $S^3$. By Remark~\ref{rem-infinity}
the resulting surface is contained in a complex projective algebraic surface covered by at least $2$-dimensional algebraic family of complex conics. By Theorem~\ref{thm:c-make-real} the surface has parametrization~\eqref{eq-circular} (for $n=2$) with real polynomials $X_1,\dots,X_{4}$ of total degree at most $2$. Take a generic real point on the surface; assume without loss of generality that it has parameters $(u,v)=(0,0)$. Then the curves $u=t v$, where $t\in\R$, are real conics, hence circles, covering an open subset of the surface. Project the surface back to $\R^3$ from our point. We get a surface containing a line segment and infinitely many circular arcs through almost every point (continuously depending on the point). By \cite[Theorem~1.1]{NS11}, the resulting surface is a subset of a quadric or a plane. Since it is covered by a 2-dimensional family of circles, by the classification of quadrics, the resulting surface, and hence the initial one, is a subset of a sphere or a plane.
\end{proof}



\begin{rem} Theorem~\ref{haupt} remains true, if the sphere $S^n$, its equation~\eqref{eq-pythagorean-n}, and circular arcs are replaced by the cylinder $\mathbb{R}\times S^{n-1}$, its equation $X_2^2+\dots+X_{n+1}^2=X_{n+2}^2$, and arcs of conics respectively. Moreover, one can conclude that the initial arcs of conics are the curves $u = \mathrm{const}$ and $v = \mathrm{const}$ in the resulting parametrization. The proof is the same except one step: To show that $\Phi$ is not covered by a $1$-dimensional family of real lines, notice that all the real lines in $\R\times S^{n-2}$ are parallel. If $\Phi$ were ruled, then it would be a cylinder with a conic in the base, hence no two conics in $\Phi$ had a unique common point, a contradiction. This version of the theorem is applied in \cite{Morozov-20}.
\end{rem}

\section{Conclusions and open problems}
\label{sec-open}

\vspace{-0.3cm}

\subsection*{Proof of main theorem}
\label{completion}


We solve the posed problem first in $4$-dimensional space and then in $3$-dimensional space.

\begin{cor}
\label{cor-4DBconj}
If a surface in $\mathbb{R}^4$ satisfies the assumptions of Theorem~\ref{haupt} with $S^n$ replaced by $\mathbb{R}^4$ then some M\"obius transformation of the surface (besides a $1$-dimensional subset) has parametrization~\eqref{eq-4DBconj}.
\end{cor}


\begin{proof}
Perform the inverse stereographic projection of $\mathbb{R}^4$ to $S^4$.
By Theorem~\ref{haupt} the resulting surface has  parametrization~\eqref{eq-circular} with $n=4$.
By Theorem~\ref{th-pythagorean} up to a linear transformation preserving Eq.~\eqref{eq-pythagorean-n} we have Eq.~\eqref{eq-param-pythagorean}.
Then the stereographic projection of the surface given by~\eqref{eq-circular} and~\eqref{eq-param-pythagorean} is a M\"obius transformation of the
initial surface and has parametrization~\eqref{eq-4DBconj}.
\end{proof}

\begin{exmp}\label{ex-4D} The surface
$
\Phi(u, v) = (u + i)(v + j)^{-1}
$
in $\mathbb{R}^4$ contains a line and infinitely many circles through each point but is not a quadric, in contrast to an analogous result in $\mathbb{R}^3$~\cite[Theorem~1.1]{NS11}.
\end{exmp}

\begin{rem}\label{rem-Kollar} By~\cite[Theorems~3,8 and Propositions~11,12]{Kollar-16} all surfaces in $\mathbb{R}^4$ containing \emph{infinitely many} circles through each point (not considered in Corollary~\ref{cor-4DBconj}) are either spheres or M\"obius transformations of Example~\ref{ex-4D}. In particular, Corollary~\ref{cor-4DBconj} is still true for such surfaces.
\end{rem}

\begin{proof}[Proof of Theorem~\ref{mainthm}]
Consider $\mathbb{R}^3$ as a subset of $\mathbb{R}^4$. Consider the following $3$ cases.

First assume that the two circular arcs drawn through each point of the surface are cospheric. Then the surface is a subset of~(D) by Theorem~\ref{thm:cospheric} proved in the second appendix.

Second assume that there is an open subset of the surface such that through each point of the subset one can draw infinitely many (not nested) circular arcs contained in the surface.  Then
the surface is a subset of~(D) by Lemma~\ref{l-infinite}. 
%

Otherwise an open subset of the surface satisfies the assumptions of Theorem~\ref{haupt} with $S^n$ replaced by $\mathbb{R}^3$. By Corollary~\ref{cor-4DBconj} a M\"obius transformation of the subset has paramerization~\eqref{eq-4DBconj}. By Theorem~\ref{cor-3D} and Lemma~\ref{l-analytic} the surface is a M\"obius transformation of a subset of (C), (D), or~(E).
\end{proof}

\vspace{-0.6cm}

\subsection*{Open problems}

Variations of the initial question with some additional restrictions on the surfaces are also interesting.

\begin{problem}
Let $\alpha$, $r$, and $R$ be fixed. Find all surfaces in $\mathbb{R}^3$ such that through each point of the surface one can draw two transversal circlular arcs fully contained in the surface and

\noindent(1) having radii $r$ and $R$; or (2) intersecting at angle $\alpha$; or (3) the planes of which intersect at angle~$\alpha$.
\end{problem}

\begin{problem} Find all surfaces in $\mathbb{R}^3$ containing two parabolas with vertical axes through each point.
\end{problem}

\begin{problem} Find all surfaces in $\mathbb{R}^4$ containing a line segment and a circular arc through each point.
\end{problem}




In our main results the analyticity is not really a restriction. In fact, by \cite[Theorem~3.7]{Guth-Zahl-15} even a sufficiently large grid of circular arcs must already be a subset of an analytic surface containing $2$ circles through each point, hence a subset of (C), or (D), or (E). Here an \emph{$n\times n$ grid of arcs} is two collections of $n+1$ disjoint arcs such that each pair of arcs from distinct collections intersects. The following ``curved chessboard problem'' is the strongest possible form of Theorem~\ref{mainthm}.

\begin{problem}\label{pr-curved-chessboard} Is each $8\times 8$ grid of circular arcs 
contained in one of sets (C), (D), (E)?
\end{problem}

As a corollary, one could get the following incidence result (A.~Bobenko, private communication).

\begin{problem} Ten blue and ten red disjoint circles are given in $\mathbb{R}^3$. Each variegated pair except one has a unique intersection point. Is it true that the latter pair must have a unique intersection point?
\end{problem}

It is natural to ask for an analogue of simple parametrization~\eqref{eq-4DBconj} for surfaces in $\mathbb{R}^3$ rather than $\mathbb{R}^4$.

\begin{exmp}
Let $M=(M_{ij})$ be a skew-hermitian $3\times 3$ matrix such that $M_{22}\in \mathbb{H}_{10}$, $M_{33}\in \mathbb{H}_{01}$, and all the other entries $M_{ij}\in \mathbb{H}$. The quasideterminant $|M|_{11}(u,v)$ defined in \cite[Example~1.2.3.b]{Gelfand-etal-05}
is fraction-linear in $u$ and in $v$ by \cite[Theorem~1.4.2(i)]{Gelfand-etal-05}.
By Lemma~\ref{l-axial} $\Phi(u,v)=|M|_{11}(u,v)$ is a surface in $\mathbb{R}^3$ containing $2$ circles $u=\mathrm{const}$ and $v=\const$ through each point, for almost every $M$.
\end{exmp}

\begin{problem}
Can each surface in $\mathbb{R}^3$ containing $2$ noncospheric circles through each point be parametrized as a quasideterminant of a skew-hermitian matrix with the entries from $\mathbb{H}[u,v]$?
\end{problem}

One of our results (Corollary~\ref{cor-reducibility}) leads to a conjecture that unique factorization holds in a sense for quaternionic polynomials of degree 1 in one of the two variables. Let us make it precise (cf.~\cite{Ore}).

\begin{problem} Two decompositions of a polynomial from $\mathbb{H}[u,v]$ into irreducible factors of degree $\le 1$ in $v$ are given. Is it true that the factors of the two decompositions are similar in pairs?
\end{problem}

\begin{problem} \label{pr-deg2} Do Lemmas~\ref{l-splitting4} and~\ref{th-linear-PDT} remain true for $P,Q,R$ of entire degree $2$ and $3$ respectively?
\end{problem}

\begin{problem}
Does Splitting Lemma~\ref{l-splitting-basic} remain true for polynomials in more than $2$ variables?
\end{problem}

Although our results are stated for quaternionic polynomials, they seem to reflect a general algebraic phenomenon. 
The latter may be useful to answer our geometric question in higher dimensions.

\begin{problem} Do \ref{l-adding variable2},\ref{l-splitting-basic},\ref{l-splitting4},
\ref{cor-reducibility}
remain true, if  
$\mathbb{H}$ is replaced by another ring with
an~involution?
\end{problem}


\subsection*{Acknowledgements}

This results have been presented at Moscow Mathematical Society seminar, SFB ``Discretization in geometry and dynamics'' colloquim in Berlin, Ya.~Sinai -- G.~Margulis conference, and the conference ``Perspectives in real geometry'' in Luminy. The authors are especially grateful to A.~Pakharev for joint numerical experiments, related studies \cite{Pakharev-Skopenkov-15}, and
for pointing out that numerous surfaces we tried to invent all have form~\eqref{eq-4DBconj}. The authors are grateful to N.~Lubbes and L.~Shi for parts of
Video~\ref{movie}, to
A.~Gaifullin and S.~Ivanov for finding gaps in earlier versions of the proofs, to A.~Bobenko, J.~Capco, S.~Galkin, A.~Kanunnikov, O.~Karpenkov, A. Klyachko, J.~Koll\'ar, W.~K\"uhnel, A.~Kuznetsov, E.~Morozov, N.~Moshchevitin, S.~Orevkov, F.~Petrov, R.~Pignatelli, F. Polizzi, H. Pottmann,  G.~Robinson, I.~Sabitov, J.~Schicho, K.~Shramov, S.~Tikhomirov, V.~Timorin, M.~Verbitsky, E.~Vinberg, J.~Zahl, S.~Zub\.e for useful discussions. The first author is grateful to King Abdullah University of Science and Technology for hosting him during the start of the work over the paper.



\section*{Appendix. A rationality result}

In this appendix we give a direct proof of Lemma~\ref{l-rational} going along the lines of \cite{schicho:2001} and exposed for nonspecialists.
We use notation of Section~\ref{sec:proofs2}, in particular, from the statements of lemmas there.

A projective algebraic surface $\Psi$ is \emph{unirationally ruled} (or simply \emph{uniruled}), if there are an algebraic curve $\eta$ and a rational map $\eta\times P^1\dasharrow\Psi$ whose image is the whole $\Psi$ possibly besides a $1$-dimensional set. A surface $\Psi$ is \emph{birationally ruled} (or simply \emph{ruled}), if there is a birational map $\eta\times P^1\dasharrow\Psi$. A curve $\eta$ is \emph{rational}, if there is a birational map $P^1\dasharrow\eta$.

\begin{lemma} \label{l-uniruled} The surface $\hat\Phi$ is unirationally ruled.
\end{lemma}

\begin{proof} Let $\hat\gamma$ be the irreducible curve in the variety of conics given by Lemma~\ref{l-algebraic}. Consider the algebraic set $\hat\Psi := \{(P, \alpha)\in \hat\Phi\times\hat\gamma : P\in\alpha\}. $ The second projection $\hat\Psi\to \hat\gamma$ is a rational map such that a generic fiber is a conic (hence a rational curve). By the Noether--Enriques theorem \cite[Theorem~III.4]{Beauville}  $\hat\Psi$ is birationally ruled. In particular there is an algebraic curve $\eta$ and a rational map $\eta\times P^1\dasharrow \hat\Psi$, whose image is the whole $\hat\Psi$ besides a $1$-dimensional set. Compose the map with the first projection $\hat\Psi\to\hat\Phi$. This projection is surjective by \cite[Theorem~2 in \S I.5.2]{Shafarevich} because the image contains the $2$-dimensional subset $\bigcup_v\alpha_v\cap\Omega=\Phi\cap\Omega$ by Lemma~\ref{l-cover}. We get a rational map $\eta\times P^1\dasharrow \hat\Phi$, whose image is $\hat\Phi$ besides a $1$-dimensional set, i.e., $\hat\Phi$ is unirationally ruled.
\end{proof}


\begin{lemma} \label{l-ruled} Each unirationally ruled surface is birationally ruled.
\end{lemma}

\begin{proof}
Let $\hat\Phi$ be a unirationally ruled surface and $\eta\times P^1\dasharrow\hat\Phi$ be a rational map whose image is the whole $\hat\Phi$ besides a $1$-dimensional set.
Take a desingularization $d:\tilde\Phi\to\hat\Phi$ defined in the 1st paragraph of the proof of Lemma~\ref{l-linear-pencil}.
Let $\hat\Phi\dasharrow\tilde\Phi$ be the inverse rational map of the desingularization.

Consider the rational map $\eta\times P^1\dasharrow \hat\Phi\dasharrow \tilde\Phi$. By the theorem on eliminating indeterminacy \cite[Theorem II.7]{Beauville} this rational map equals a composition $\eta\times P^1\dasharrow \tilde\Psi\to \tilde\Phi$, where  $\tilde\Psi$ is a smooth projective algebraic surface, the first map is birational, and the second map is rational and defined everywhere. In particular, $\tilde\Psi$ is birationally ruled. Since the image $\eta\times P^1\dasharrow \hat\Phi$ is the whole $\hat\Phi$ besides a $1$-dimensional set and the surfaces are compact, by \cite[Theorem~2 in \S I.5.2]{Shafarevich} it follows that $\tilde\Psi\to \tilde\Phi$ is surjective.

By the Enriques theorem \cite[Theorem~VI.17 and Proposition~III.21]{Beauville}, a smooth projective algebraic surface is birationally ruled if and only if for each $k>0$ the $k$-th tensor power of the exterior square of the cotangent bundle has no sections except identical zero (in other terminology, \emph{the surface has Kodaira dimension} $-\infty$, or \emph{all plurigeni vanish}). Assume, to the contrary, that $\tilde\Phi$ has such a section (\emph{pluricanonical section}). Then the pullback under the surjective rational map $\tilde\Psi\to\tilde\Phi$ is such a section for the birationally ruled surface $\tilde\Psi$, a contradiction. Thus $\tilde\Phi$, and hence $\hat\Phi$, is birationally ruled.
\end{proof}


\begin{proof}[Proof of Lemma~\ref{l-rational}]
By Lemmas~\ref{l-uniruled} and~\ref{l-ruled} the surface $\hat\Phi$ is birationally ruled. Thus there is a birational map $\hat\Phi\dasharrow\eta\times P^1$. Consider the two conics through a generic point of the surface $\hat\Phi$. Their images are two distinct rational curves through a point of $\eta\times P^1$. Since there is only one $P^1$-fiber through each point, at least one of the rational curves is nonconstantly projected to the curve $\eta$. By the L\"uroth theorem \cite[Theorem~V.4]{Beauville} the curve $\eta$ must be rational, and hence $\hat\Phi$  is rational.
\end{proof}

\begin{rem} We conjecture that the following generalization of Lemma~\ref{l-rational} is true: \emph{an algebraic surface containing two rational curves through almost every point is rational}. See~\cite[Definition~IV.3.2, Theorems~ IV.5.4, IV.3.10.3, IV.2.10, Corollary~IV.5.2.1, Exercise~IV.3.12.2]{Kollar} for a sketch of the proof.
\end{rem}



\section*{Appendix. Surfaces with infinitely many conics through each point}


\footnotetext{This appendix is not present in the published version of this paper; it fills a gap there and has been submitted as an erratum.}

\begin{flushright}
\textit{It may be true that there is no subject in mathematics upon which more\\
unintelligible nonsense has been written than that of families of curves\\
R.\,P.\,Agnew, 1960}
\end{flushright}

In this appendix we prove auxiliary Theorems~\ref{thm:c-make-real} and~\ref{thm:cospheric} stated below. The former is essentially due to J.\,Schicho \cite{schicho:2001} (cf. the work by J.\,Koll\'ar~\cite{Kollar-16}), and the latter --- to J.L.~Coolidge \cite{coolidge:1916}.

\begin{thm}\label{thm:c-make-real}
Assume that through each point of an analytic surface $\Phi$ in a domain in ${P}^n$ one can draw infinitely many complex conics such that their intersections with the domain are contained in the surface. Assume that among the infinite number, one can choose two transversal conics analytically depending on the point. Then the surface (possibly besides a one-dimensional subset) has a parametrization~\eqref{eq-conical} for some $X_1,\dots,X_{n+1}\in\mathbb{C}[u,v]$ of \emph{total} degree~$2$. Moreover, if the set of real points of $\Phi$ is $2$-dimensional and is not covered by a $1$-dimensional family of real lines, then there is parametrization~\eqref{eq-conical} for some \emph{real} $X_1,\dots,X_{n+1}\in\mathbb{R}[u,v]$ of total degree~$2$.
\end{thm}

The assumption that $\Phi$ is not ruled has been added to simplify the proof; it does not seem to be essential.

A \emph{real conic} is a conic with infinitely many real points, and also the set of real points itself.

\begin{exmp}\label{exm:unreal-cone}
The surface $\Phi=\{(x,y,z)\in\mathbb{C}^3:x^2+y^2+z^2=0\}$ contains infinitely many complex (but not real) conics through each point except the origin and cannot be parametrized by real polynomials. This shows that the assumption on the dimension of the set of real points is essential.
\end{exmp}

\begin{exmp}[{\cite[Remark~4]{schicho:2001}}]
The surface $\Phi=\{(x,y,z)\in\mathbb{C}^3: yx^2+z^2+1=0\}$ has parametrization~\eqref{eq-conical} by \emph{complex} polynomials $X_1,\dots,X_{4}\in\mathbb{C}[u,v]$ of degree at most $2$ in \emph{each} of the variables $u$ and $v$, but cannot be parametrized by \emph{real} polynomials of degree at most $2$ in each of the variables $u$ and $v$. The surface contains two families of real conics $x=\const$ and $y=\const$.
\end{exmp}

We prove Theorem~\ref{thm:c-make-real} by exhaustion, using the following classification of surfaces containing two conics through each point due to J.\,Schicho \cite{schicho:2001}. One of the classes is del Pezzo surfaces, but we omit their definition because that class is going to be ruled out soon.

A map between algebraic subsets $X\subset P^d$ and $Y\subset P^n$ is \emph{regular}, if in a neighborhood of each point of $X$, it is given by polynomials in homogeneous coordinates.
An \emph{antiregular} map is the composition of a regular map with the complex conjugation $^*$ on $P^d$. A \emph{projective} map $P^d\dashrightarrow P^n$ is given by linear polynomials in homogeneous coordinates. 
For other 
terminology we refer to~\cite[\S2]{schicho:2001}.


\begin{thm}\label{thm:classification}
Assume that through each point of an analytic surface $\Phi$ in a domain in ${P}^n$ one can draw two transversal conics intersecting each other only at this point (and analytically depending on the point) such that their intersections with the domain are contained in the surface. Then the surface $\Phi$ is the image of a subset of one of the following algebraic surfaces $\overline{\Phi}$ under a projective map $N\colon P^d\dashrightarrow P^n$:
\begin{enumerate}
\item\label{item:quadric} a quadric or a plane in ${P}^3$ ($d=3$); or
\item\label{item:ruled} (the closure of) the ruled surface $\overline{\Phi}(u,v)=(1:u:v:u^2:uv)$ in ${P}^4$ ($d=4$); or
\item\label{item:veronese} (the closure of) the \emph{Veronese surface} $\overline{\Phi}(u,v)=(1:u:v:u^2:uv:v^2)$ in ${P}^5$ ($d=5$); or
\item\label{item:del-pezzo} a \emph{del Pezzo surface}, i.~e., a linearly normal degree $d$ surface in ${P}^d$. 
\end{enumerate}
The restriction $N\colon\overline{\Phi}\to N\overline{\Phi}$ is
a regular birational map preserving the degree of the surface and the degree of any curve with the image not contained in the set of singular points of $N\overline{\Phi}$.
\end{thm}

For $n=3$ this is a straightforward consequence of \cite[Theorems~5-7 and Proposition~1]{schicho:2001}. For $n>3$ one cannot formally apply those results, because they stated for surfaces in ${P}^3$ rather than in ${P}^n$. But one can perform a simple reduction as follows.

\begin{proof}[Proof of Theorem~\ref{thm:classification}]
The surface $\Phi$ is contained in an algebraic one by Lemma~\ref{l-algebraic} and Remark~\ref{rem-infinity}. Let $\overline{\Phi}$ be its linear normalization defined in \cite[\S2]{schicho:2001}. By that definition, $\overline{\Phi}$ is a linear normalization of a generic projection of $\Phi$ to ${P}^3$ as well. The projection satisfies all the assumptions of Theorem~\ref{thm:classification} for $n=3$. Then by \cite[Theorem~5 and Proposition~1]{schicho:2001} $\overline{\Phi}$ is one of surfaces \ref{item:quadric}--\ref{item:del-pezzo}. By \cite[Theorem~6]{schicho:2001} $\Phi$ is the image of (a subset of) $\overline{\Phi}$ under a projective map, and the map has the required properties.
\end{proof}


\begin{lem}\label{lem:del-pezzo}
The image of a del Pezzo surface under a projective map (preserving the degrees of curves with the images not contained in the set of singular points) cannot satisfy the assumptions of Theorem~\ref{thm:c-make-real}.
\end{lem}

\begin{proof}
A del Pezzo surface has degree $d\le 9$ \cite[\S4]{schicho:2001}, and by
\cite[Proposition~1 and Theorem~10]{schicho:2001} it contains finitely many pencils of conics. The number of conics on the surface not belonging to pencils is at most countable (e.~g., because the class group of $\overline{\Phi}$ is finitely generated \cite[\S4]{schicho:2001}). Thus there are only finitely many conics through a generic point of a del Pezzo surface, and hence through a generic point of its image, because the degrees of curves are preserved.
\end{proof}

\begin{proof}[Proof of the first part of Theorem~\ref{thm:c-make-real}]
By Theorem~\ref{thm:classification} and  Lemma~\ref{lem:del-pezzo}, $\Phi$ is the projective image of one of the surfaces \ref{item:quadric}---\ref{item:veronese}. These surfaces, hence their image $\Phi$, have required parametrization~\eqref{eq-conical} for some $X_1,\dots,X_{n+1}\in\mathbb{C}[u,v]$ of total degree at most $2$.
\end{proof}

Now let us prove the ``moreover'' part of the theorem. Although this could be deduced from general theory of linear normalization over $\R$, we could not find any reference suitable for direct citation; see the following example and \cite{bib:mathoverflow}.

\begin{exmp}\label{exm:normalization}
The cone ${\Phi}=\{(x,y,z)\in\mathbb{C}^3:x^2+y^2-z^2=0\}$ is its own real linear normalization, but there is a complex linear normalization map ${\Phi}\to\Phi$, $(x,y,z)\mapsto(z,iy,x)$, which is not given by real polynomials. (This is indeed a complex linear normalization map because the latter is only defined up to a projective isomorphism.) This shows that the complex linear normalization map $\overline{\Phi}\to\Phi$ in Theorem~\ref{thm:classification} may not be given by real polynomials. Using linear normalization over $\R$, one may wish to find \emph{another} surface $\overline{\Phi}_\R$ such that the map $\overline{\Phi}_\R\to\Phi$ is given by real polynomials; but then it becomes unclear why $\overline{\Phi}_\R$ itself can be parametrized by \emph{real} polynomials of degree~$2$.
\end{exmp}

Thus we present an alternative proof. The only interesting case is when $\Phi$ is a projective image of the Veronese surface~\ref{item:veronese}; we are going to show that the preimage of the real points can be made real by a complex automorphism of the surface. We need several well-known lemmas, which could not find a reference for; see~\cite{bib:mathoverflow}. 


\begin{lem}\label{lem:conjugation}
If an irreducible algebraic surface $\Phi\subset {P}^n$ has $2$-dimensional set of real points, then $\Phi^*=\Phi$.
\end{lem}

\begin{proof} Assume that $\Phi^*\ne\Phi$. Since $\Phi$ is irreducible, it follows that $\Phi^*\cap\Phi$ is an algebraic curve (or is finite). But $\Phi^*\cap\Phi$ contains the $2$-dimensional set of real points of $\Phi$. Thus the intersection of the algebraic curve $\Phi^*\cap\Phi$ with a generic real hyperplane in ${P}^n$ is infinite, a contradiction.
\end{proof}

\begin{lem}\label{lem:veronese}
The \emph{Veronese} map $V\colon P^2\to P^5$, $(u:v:w)\mapsto (w^2:wu:wv:u^2:uv:v^2)$, has a regular inverse $V^{-1}\colon V(P^2)\to P^2$. The \emph{Veronese surface} $V(P^2)$ is smooth and contains no lines.
\end{lem}

\begin{proof} The inverse map is given by one of the linear polynomials $V^{-1}(x_1:\dots:x_6)=(x_2:x_3:x_1)$, or $(x_4:x_5:x_2)$, or $(x_5:x_6:x_3)$, depending on which of the triples is nonzero. The image of a line $l$ on $V(P^2)$, if any, would be a line $V^{-1}(l)$ on $P^2$. Parametrize $V^{-1}(l)$ by three coprime polynomials $(u(t):v(t):w(t))$ of degree $\le 1$. Then $u^2(t)+v^2(t)+w^2(t)$ has no real roots, hence has exactly two distinct complex conjugate roots. Thus the hyperplane $x_1+x_4+x_6=0$ in $P^5$ intersects $l$ at exactly $2$ points. Thus $l$ is not a line, a contradiction.
Since $V^{-1}$ is smooth, it follows that the differential $dV$ is nondegenerate, hence $V(P^2)$ is smooth.
\end{proof}

\begin{lem}\label{lem:exceptional} 
Let $N\colon P^5\dashrightarrow P^n$ be a projective map
taking an algebraic surface $\overline{\Phi}\subset P^5$ containing no lines to a surface of the same degree. Then $\overline{\Phi}$ does not contain \emph{exceptional} curves, i.~e., curves mapped to points.
\end{lem}

\begin{proof}
Without loss of generality assume that $N\colon {P}^5\dashrightarrow {P}^n$ is a projection to a projective subspace, where $n=3$ or $4$. The preimage of a point under the restriction $N\colon\overline{\Phi}\to {P}^n$ is the intersection with a $(5-n)$-dimensional projective subspace containing the $(4-n)$-dimensional projection center.

Assume that the preimage is $1$-dimensional. Then by the Bezout theorem its closure intersects the projection center. Hence $\overline{\Phi}$ intersects the projection center. Take a generic $3$-dimensional subspace $H\subset P^5$ passing through the projection center. The number of intersection points of $H$ with $N\overline{\Phi}$ is the degree of $N\overline{\Phi}$. The number of intersection points of $H$ with $\overline{\Phi}$ is greater, but still finite in general position because $\overline{\Phi}$ does not contain lines (hence cannot contain the whole $1$-dimensional projection center). Thus $\overline{\Phi}$ cannot have the same degree as $N\overline{\Phi}$, a contradiction.
\end{proof}




\begin{lem}[Cf.~\cite{bib:mathoverflow}]\label{lem:lift}
Let $N\colon P^d\dashrightarrow P^n$ be a projective map such that the restriction $N\colon \overline{\Phi}\to N\overline{\Phi}$ to a smooth algebraic surface $\overline{\Phi}\subset P^d$ is a regular birational map. Assume that $\overline{\Phi}$ does not contain exceptional curves, $\overline{\Phi}=\overline{\Phi}^*$, and $N\overline{\Phi}=(N\overline{\Phi})^*$.
Then the complex conjugation on ${P}^n$ lifts to an antiregular involution on $\overline{\Phi}$ so that the involutions commute with the map $N$.
\end{lem}

\begin{exmp}
The example $N\colon (u:v:w)\mapsto (iu:v:w)$, $d=n=2$, shows that the required involution may not come from the complex conjugation on ${P}^d$.
\end{exmp}

\begin{proof}[Proof of Lemma~\ref{lem:lift}]
First we lift to the nonsingular points, and then extend to the singular ones.

Let $\Sigma\subset\overline{\Phi}$ be the preimage of the singular points of the surface $N\overline{\Phi}$ under the map $N$. It is either an algebraic curve or a finite set because $N\colon\overline{\Phi}\to N\overline{\Phi}$ is birational.
Since $\overline{\Phi}$ does not contain exceptional curves, by \cite[Chapter~II, Theorem~2 in \S4.4]{Shafarevich} it follows that the inverse $N^{-1}$ of the birational map $N\colon \overline{\Phi}\to N\overline{\Phi}$ is regular outside the singular points of $N\overline{\Phi}$ (cf.~\cite[Proof~of~Theorem~6]{schicho:2001}). Thus the composition $f:=N^{-1}\circ N^*$ is a well-defined antiregular involution on $\overline{\Phi}-\Sigma$, because the set of singular points of $N\overline{\Phi}$ is invariant under the complex conjugation.

Consider the regular map $f^*\colon\overline{\Phi}-\Sigma\to f^*(\overline{\Phi}-\Sigma)$. It has an inverse map $(f^*)^{-1}\colon f^*(\overline{\Phi}-\Sigma)\to \overline{\Phi}-\Sigma$. Since $\overline{\Phi}$ is smooth, by \cite[Chapter~II, Theorem~3 in \S3.1]{Shafarevich} the maps uniquely extend to birational maps $f^*$ and $(f^*)^{-1}$ regular on the whole $\overline{\Phi}$ possibly except a finite set.

For each point $P\in\overline{\Phi}$ where $f^*$ is regular we have $N(f(P))=N^*(P)$. Indeed, take a sequence $P_n\to P$, $P_n\in\overline{\Phi}-\Sigma$. By the definition of $f$, we have $N(f(P_n))=N^*(P_n)$. By the continuity of $f^*$ at $P$, we get $N(f(P))=N^*(P)$.

Now assume that $(f^*)^{-1}$ is not regular at (i.~e., not extendable to) a point $Q\in\overline{\Phi}$. Then by \cite[Lemma~II.10]{Beauville} there is a curve $\gamma\subset\overline{\Phi}$ such that $f^*(\gamma)=Q$. By the assumptions of the lemma, $N(\gamma)$ is not a point, hence $N^*(\gamma)\ne N(Q^*)$. Then there is $P\in\gamma$ such that $N^*(P)\ne N(Q^*)$ and $f^*$ is regular at $P$. By the previous paragraph, $N^*(P)=N(f(P))=N(f(\gamma))=N(Q^*)$, a contradiction. Thus $(f^*)^{-1}$, analogously $f^*$, hence $f$ and $f^{-1}$, are regular on the whole $\overline{\Phi}$. By the continuity, the extended antiregular map $f$ is still an involution.
\end{proof}

\begin{lem}\label{lem:involution} 
Each antiregular involution of ${P}^{2m}$ is projectively conjugate to the complex conjugation. Each antiregular involution of ${P}^{2m-1}$ is projectively conjugate to either the complex conjugation or the involution $(x_1:\dots:x_{2m})\mapsto (-x_{m+1}^*:\dots:-x_{2m}^*:x_1^*:\dots:x_{m}^*)$.
\end{lem}


\begin{proof}
First let us show that a regular map $f\colon{P}^{d}\to {P}^{d}$ having a regular inverse is projective (N.~Lubbes, private communication). Indeed, for $d=1$ this follows from the fundamental theorem of algebra. For $d\ge 2$ take a line $l$ and a hyperplane $h$ intersecting transversely at a unique point $P$. By \cite[Chapter~II, Corollary~2 in \S1.3]{Shafarevich}, it follows that $f(l)$ and $f(h)$ are algebraic surfaces intersecting transversely at the unique point $f(P)$. By the Bezout theorem, $f(l)$ is a line. Since $l$ is arbitrary, by the fundamental theorem of projective geometry (by M\"obius--von Staudt), $f$ is projective.

Thus any antiregular involution of ${P}^{d}$ is a composition of a projective transformation and the complex conjugation, i.~e., it is the projectivization of the map $v\mapsto Lv^*$ for some complex $(d+1)\times (d+1)$ matrix $L$. E.~g., for the identity matrix $L=\id$ the map is complex conjugation, and for the $2m\times 2m$ matrix $L=J:=\left(\begin{smallmatrix} 0 & -\id\\ \id & 0\end{smallmatrix}\right)$ it is $(x_1:\dots:x_{2m})\mapsto (-x_{m+1}^*:\dots:-x_{2m}^*:x_1^*:\dots:x_{m}^*)$.

The map $v\mapsto Lv^*$ is an involution, if and only if for each $v\in\mathbb{C}^{d+1}$ there is $k(v)\in\mathbb{C}$ such that $L(Lv^*)^*=k(v)v$ . Thus $LL^*=k\id$ for some $k\in\mathbb{C}$, where $L^*$ denotes the complex conjugate matrix (not to be confused with the commonly used notation for the conjugate transpose matrix). Thus $L$ and $L^*$ commute and $k\id=LL^*=L^*L=(LL^*)^*=k^*\id$, hence $k\in\mathbb{R}$. Dividing $L$ by $\sqrt{|k|}$, we arrive at the equation $LL^*=\pm\id$. Passing to determinants, we get $\det L (\det L)^*=(\pm 1)^{d+1}$, hence the minus sign is impossible for even $d$.

Projective transformations of ${P}^{d}$ act on the solutions of the equations $LL^*=\pm\id$ by the transformations $L\mapsto U^{-1}LU^*$, where $U$ is an invertible complex $(d+1)\times (d+1)$ matrix. It remains to show that each solution can be transformed to $\id$ and $J$ respectively, i.~e., each matrix $L$ satisfying $LL^*=\id$ (respectively, $LL^*=-\id$) has form $L=U^{-1}U^*$ (respectively, $L=U^{-1}JU^*$ ) for some invertible matrix $U$.

If $LL^*=\id$, then take $a\in\mathbb{R}$ such that $-e^{2ia}$ is not an eigenvalue of $L^*$. Then $U:=e^{ia}\id+e^{-ia}L^*$ is the required invertible matrix because $UL=(e^{ia}\id+e^{-ia}L^*)L=e^{ia}L+e^{-ia}\id=U^*$ by the equation $L^*L=LL^*=\id$.

If $LL^*=-\id$, then  take $a\in\mathbb{R}$ such that $e^{2ia}$ is not an eigenvalue of $JL^*$. Then $U:=e^{ia}J+e^{-ia}L^*$ is the required matrix: $UL=(e^{ia}J+e^{-ia}L^*)L=e^{ia}JL+e^{-ia}JJ^*=JU^*$ because $L^*L=LL^*=-\id=JJ^*$.
\end{proof}

\begin{proof}[Proof of the `moreover' part of Theorem~\ref{thm:c-make-real}]
By Theorem~\ref{thm:classification} and Lemma~\ref{lem:del-pezzo}, $\Phi$ is the projective image of one of the surfaces \ref{item:quadric}--\ref{item:veronese}.

\textit{Case \ref{item:quadric}.} Here ${\Phi}$ is a quadric or a plane.
Clearly, if the set of real points of $\Phi$ is $2$-dimensional, then it has a parametrization~\eqref{eq-conical} with real $X_1,\dots,X_{n+1}$: it is the inverse of a projection of a quadric from a point on it.

\textit{Case \ref{item:ruled}.} Here through a generic point of $\Phi$ there passes a unique line $u=\mathrm{const}$ contained in the surface. The line through a generic real point of $\Phi$ must be real, otherwise a complex conjugate line would be the second one. Thus the case is ruled out by the assumptions of the theorem.

\textit{Case \ref{item:veronese}.} By Lemmas~\ref{lem:conjugation}--\ref{lem:lift}, there  is an antiregular involution on ${P}^2$, taken to the complex conjugation of ${P}^n$ by $N\circ V\colon {P}^2\to {P}^n$. Lemma~\ref{lem:involution} implies that after an appropriate complex projective change of the coordinates $u:v:w$,
the map $N\circ V$ takes all real points $(u:v:w)$ of $P^2$ to real points of ${P}^n$. Let the resulting map be given by the polynomials $X_1,\dots,X_{n+1}$; they have total degree $2$. By Lemma~\ref{l-real}, after canceling a common factor, the polynomials 
become real.
\end{proof}

Now we prove the following result (adding the details to \cite[Sketch of the proof of Theorem~3.5]{NS11}.

\begin{thm}[{Cf.~\cite[Theorem~3.5]{NS11}}]
\label{thm:cospheric}
Assume that through each point of an analytic surface in $\R^3$ one can draw two cospheric circular arcs fully contained in the surface and analytically depending on the point. Then the surface is a subset of a Darboux cyclide.
\end{thm}

\begin{proof}[Proof of Theorem~\ref{thm:cospheric}]
First perform the inverse stereographic projection to $S^3$ and denote by $\Phi$ the resulting surface. Then perform an analytic extension (cf.~the proof of Lemma~\ref{l-reduction}: extend an appropriate part of $\Phi$ to a complex analytic surface $\hat\Phi$ in a domain $\Omega\subset P^4$, and extend the real analytic families $\alpha_P$ and $\beta_P$ of real conics on $\Phi$ to complex analytic families $\hat\alpha_P$ and $\hat\beta_P$ of complex conics in ${P}^4$ parametrized by a point $P$ of $\hat\Phi$.

If through each point of $\hat\Phi$ one can draw infinitely many complex conics such that their intersections with $\Omega$ are contained in $\hat\Phi$, then by Theorems~\ref{thm:c-make-real} and~\ref{l-infinite} the desired result follows. Otherwise by Lemma~\ref{l-cover} the surface $\hat\Phi$ is covered by families of conics $\alpha_v$, $\beta_u$ as in the lemma. With these families at hand, we proceed by consideration of real surface $\Phi$ and just real families members, without need to worry about complex ones anymore.

Recall that spheres and circles in $S^3$ are its sections by 3- and 2-planes. By a \emph{generalized sphere} (respectively, \emph{generalized circle}) we mean a real projective 3-plane (respectively, 2-plane) in real projective space $\mathbb{R}P^4$ (not necessarily intersecting $S^3$). Generalized spheres $a_1x_1+\dots+a_5x_5=0$ and $b_1x_1+\dots+b_5x_5=0$, where $x_1:\dots:x_5$ are homogeneous coordinates in $\mathbb{R}P^4$, are \emph{orthogonal}, if  $a_1b_1+\dots+a_4b_4-a_5b_5=0$. A generalized sphere orthogonal to each sphere containing a fixed generalized circle is called \emph{orthogonal} to the circle. Clearly, for each two generalized circles there is a generalized sphere orthogonal to both of them.

Thus take two sufficiently close circles $\alpha_1$ and $\alpha_2$ in the family $\alpha_v$, and fix a generalized sphere $\Sigma$ orthogonal to them. Assume without loss of generality that each circle $\beta_u$ intersects both $\alpha_1$ and $\alpha_2$. Since the circles $\beta_u$ and $\alpha_1$ (respectively, $\alpha_2$) passing through the point $P(u,1)$ are cospheric, it follows that $\alpha_1\cup\beta_u$ and $\alpha_2\cup\beta_u$ are contained in some spheres $\Sigma_{1u}$ and $\Sigma_{2u}$ respectively. Since $\alpha_1 \perp \Sigma$ it follows that $\Sigma_{1u} \perp \Sigma$. Analogously,
$\Sigma_{2u} \perp \Sigma$. If $\Sigma_{1u}=\Sigma_{2u}$ for each $u$ then $\Phi$ is a sphere through 
$\alpha_1$ and $\alpha_2$. Otherwise $\beta_u\subset\Sigma_{1u}\cap\Sigma_{2u} \perp \Sigma$ for each $u$. Analogously $\alpha_v\perp \Sigma$ for each $v$.

Let the equation of $\Sigma$ be $a_1x_1+\dots+a_5x_5=0$. Then all the generalized circles orthogonal to $\Sigma$ pass through the pole $\Sigma^\perp:=a_1:\dots:(-a_5)$. Thus a projection from the pole $\Sigma^\perp$ to a hyperplane takes $\alpha_v$ and $\beta_u$ to line segments, and hence takes $\Phi$ to a doubly ruled surface, i.e. a quadric. Thus $\Phi$ is contained in the intersection of $S^3$ with a quadratic cone with the vertex $\Sigma^\perp$. The initial surface, which is the stereographic projection of $\Phi$, is then a Darboux cyclide.
\end{proof}

Finally we give a remark showing that the reciprocal of Theorem~\ref{mainthm} holds besides a few explicitly listed exceptions and thus indeed solves the problem stated in the first line of the paper.

\begin{rem}\label{rem-reciprocal}
Let $\Phi$ be one of the sets (C),(D),(E) from Theorem~\ref{mainthm} having dimension $2$.
Then through almost every point of $\Phi$ one can draw at least two circular arcs or line segments fully contained in $\Phi$, if and only if $\Phi$ is M\"obius equivalent to neither surface
\begin{equation}\label{eq-exception}
(x^2+y^2+z^2)^2-2ax^2-2ay^2-2bz^2\pm1=0,
\quad\text{where }
\begin{cases}
a\ne b>1\ne a^2, &\text{ for the ``$+$'' sign;}\\
a\ne b,          &\text{ for the ``$-$'' sign;}
\end{cases}
\end{equation}
nor rotational ellipsoid/paraboloid/hyperboloid of $2$ sheets, nor parabolic/hyperbolic cylinder. The exceptional surfaces are exactly the rotational/translational surfaces~(D) that cannot be generated by rotation/translation of a circle or a line.
\end{rem}

\begin{proof}[Proof of Remark~\ref{rem-reciprocal}]
For the set~(E) there is nothing to prove. For the set~(C), 
the assertion follows from its equivalent definition in \S\ref{s:intro}. For the set~(D) M\"obius equivalent to a quadric, the assertion follows from the classification of quadrics up to isometry. For the set~(D) M\"obius equivalent neither to a quadric nor to surface~\eqref{eq-exception}, the assertion follows from \cite[Theorem~6.6, Proposition~3.2(i),(iii), 4.2, 4.5]{Takeuchi-00}, where the two circular arcs or line segments were explicitly constructed. To prove that remaining surfaces~\eqref{eq-exception} do not contain two circles or lines through a generic point, we apply the algorithm \cite[Remark~8]{PSS11} (this result was stated in \cite[Proposition~3.2(ii) and 4.5]{Takeuchi-00} without any indication of a proof).

Consider the case of the ``$+$'' sign in~\eqref{eq-exception}; the case of the ``$-$'' sign is similar and simpler. The inverse stereographic projection of~\eqref{eq-exception} is the intersection $\Psi$ of all quadrics in the pencil $X^T (A-tJ)X=0$, where $X=(x_1,x_2,x_3,x_4,x_5)$ are the homogeneous coordinates in $\mathbb{R}P^4$, $A=\mathrm{diag}(a,a,b,-1,-1)$, and $J=\mathrm{diag}(1,1,1,1,-1)$. By \cite[Theorem~3]{PSS11}, the circles on $\Psi$ are the sections by the $2$-planes contained in the quadrics $X^T (A-tJ)X=0$ with $\det(A-tJ)=0$. The latter equation has roots $t=a,b,\pm 1$, where $b>1$, leading to the quadrics of signatures
$$
(0,0,\mathrm{sgn}(b-a),\mathrm{sgn}(-a-1),\mathrm{sgn}(a-1)),\quad
(\mathrm{sgn}(a-b),\mathrm{sgn}(a-b),0,-,+),\quad
(\mathrm{sgn}(a\mp1),\mathrm{sgn}(a\mp1),+,-,0).
$$
The first signature is either $(0,0,-,-,+)$ or $(0,0,-,+,+)$ because $a\ne b,\pm 1$ and $\mathrm{sgn}(b-a)=\mathrm{sgn}(-a-1)=\mathrm{sgn}(a-1)$ would contradict to the assumption $b>1$. The second and the third signatures can only be $(0,-,-,-,+)$ or $(0,-,+,+,+)$ because $a\ne b,\pm 1$. Thus the quadric $X^T (A-aJ)X=0$ contains just one family of $2$-planes, whereas the other quadrics $X^T (A-tJ)X=0$ contain no $2$-planes. Hence $\Psi$ is covered by just one family of pairwise disjoint circles, as required.

The exceptional surfaces are clearly rotational/translational. Conversely, a rotational surface~(D) that cannot be generated by rotation/translation of a circle or a line, contains just one circle or line through a generic point.
\end{proof}


\footnotesize
\noindent
\textsc{Mikhail Skopenkov\\
National Research University Higher School of Economics (Faculty of Mathematics) \&\\
Institute for Information Transmission Problems of the Russian Academy of Sciences} 
\\
\texttt{mikhail$\cdot$skopenkov\,@\,gmail$\cdot$com} \quad \url{https://www.mccme.ru/~mskopenkov/}

\bigskip

\noindent
\textsc{Rimvydas Krasauskas\\
Faculty of Mathematics and Informatics, Vilnius University}\\
\texttt{rimvydas.krasauskas@mif.vu.lt}


\end{document}